\documentclass[12pt]{amsart}
\usepackage[utf8]{inputenc}
\usepackage[T1]{fontenc}
\usepackage[english]{babel}
\usepackage{amssymb,amstext}
\usepackage{mathrsfs}
\usepackage{amsfonts,bm}
\usepackage{amsmath,verbatim}
\usepackage{xcolor}

\usepackage{caption}
\usepackage{multicol}

\usepackage{amsthm}
\usepackage{enumerate}

\usepackage{geometry}
\definecolor{BlueGreen}{cmyk}{0.85,0,0.33,0}
\definecolor{dgreen}{rgb}{0.,0.6,0.}

\allowdisplaybreaks

\newtheorem{theorem}{Theorem}[section]

\newtheorem{lemma}[theorem]{Lemma}
\newtheorem{proposition}[theorem]{Proposition}

\theoremstyle{definition}

\newtheorem{definition}[theorem]{Definition}

\newtheorem{remark}[theorem]{Remark}

\numberwithin{equation}{section}

\renewcommand{\epsilon}{\varepsilon}
\renewcommand{\wp}{\mathbf{p}}
\newcommand{\eps}{\varepsilon}
\newcommand{\calC}{\mathcal{C}}
\newcommand{\calG}{\mathcal{G}}
\newcommand{\calB}{\mathcal{B}}
\newcommand{\calD}{\mathcal{D}}
\newcommand{\calR}{\mathcal{R}}

\newcommand{\calT}{\mathcal{T}}
\newcommand{\R}{\mathbb{R}}
\renewcommand{\P}{\mathbb{P}}
\newcommand{\E}{\mathbb{E}}
\newcommand{\bal}{\bm{\alpha}}
\newcommand{\bbet}{\bm{\beta}}
\newcommand{\bs}{\mathbf{s}}
\newcommand{\bfr}{\mathbf{r}}
\newcommand{\wt}{\widetilde}
\newcommand{\wh}{\widehat}
\newcommand{\bone}{\mathbf{1}}
\newcommand{\La}{\Longrightarrow}

\newcommand{\bA}{\mathbf{A}}
\newcommand{\bB}{\mathbf{B}}
\newcommand{\bM}{\mathbf{M}}

\newcommand{\bho}{{$\mathbf{H1}$}}
\newcommand{\bht}{{$\mathbf{H2}$}}
\newcommand{\bhth}{{$\mathbf{H3}$}}
\newcommand{\bhf}{{$\mathbf{H4}$}}
\newcommand{\bhfi}{{$\mathbf{H5}$}}

\newcommand\calCi{\calC_{\mathrm{int}}}
\newcommand\calCo{\calC_{\mathrm{out}}}

\DeclareMathOperator{\dist}{dist}
\DeclareMathOperator{\sgn}{sgn}
\DeclareMathOperator{\var}{Var}
\DeclareMathOperator{\arccot}{arccot}

\begin{document}

\title{Knudsen gas in flat tire} 

\author{ Krzysztof Burdzy and Carl-Erik Gauthier}

\address{Department of Mathematics, Box 354350, University of Washington, Seattle, WA 98195}
\email{burdzy@uw.edu}
\email{carlgaut@uw.edu}

\thanks{KB's research was supported in part by Simons Foundation Grant 506732. \\
CEG's research was supported by the Swiss National Foundation for Research Grant P2NEP2\_171951 }

\pagestyle{headings}

\maketitle

 	\begin{abstract} 
We consider random reflections (according to the Lambertian distribution) of a light ray in a thin variable width (but almost circular) tube. As the width of the tube goes to zero, properly rescaled angular component of the 
light ray position converges in distribution to a diffusion whose parameters (diffusivity and drift) are given explicitly in terms of the tube width.
 		
 	\end{abstract}
\section{Introduction}\

We will prove an invariance principle for  a light ray 
reflecting inside a very thin variable width (but almost circular) planar domain. The reflections are random and have the Lambertian distribution introduced in  \cite{L1760}. An alternative physical representation of the process is that of a gas molecule with a velocity so high that the effect of the gravitation is negligible.
In this alternative context, the Lambertian distribution is known as Knudsen's law, introduced in \cite{K1934}.

We will now present a (very) informal version of our main  result.
Consider a smooth function $h:\R\to[1,3]$ with period $2\pi$. For each $\eps\in(0,1/100)$, consider a planar domain  $\calD_\eps$ that is very close to a thin annulus with the center $(0,0)$ and radii close to 1, except that its width is $\eps h(\alpha)$, where $\alpha$ measures the angle along the tube in radians. Suppose that a light ray travels inside $\calD_\eps$ and reflects randomly according to the Lambertian distribution, i.e., the direction of the reflected trajectory forms an angle $\Theta$ with the inner normal to the boundary of $\calD_\eps$ and the density of $\Theta$  is proportional to $\cos \theta$. The directions of reflections are  independent. If $\bbet^\eps(t)$ denotes the angular coordinate of the light ray in the polar coordinates at time $t$ then properly rescaled process $\{\bbet^\eps(t), t\geq 0\}$ converges  in the Skorokhod topology, as $\epsilon$ goes to $0$, to the solution of
\begin{align}\label{j4.4}
dX_t =h'(X_t)dt + \sqrt{h(X_t)}dW_t,
\end{align}
where $W$ is  standard Brownian motion.

We will now discuss related results and motivation for this research.

The idea of multidimensional processes converging in distribution to a process on a lower dimensional manifold goes back at least to Katzenberger \cite{Katz91}. Roughly speaking, such convergence can be induced by a strong drift keeping multidimensional processes close to the manifold.

The reflection problem in thin domains was investigated in \cite{Jer,Zav06}. More specifically, this research was devoted to eigenfunctions of the Laplacian with Neumann boundary conditions. It was proved that when the width of the domain goes to 0, the eigenfunctions converge to those of a one-dimensional Sturm-Liouville operator. In our notation, the limiting operator could be expressed as $\Delta + \frac{h'(x)}{h(x)} \frac{d}{dx}$. This is strikingly close to \eqref{j4.4} in the following sense. We could time change the diffusion in \eqref{j4.4} so that it has the quadratic variation equal to 1. Then the time-changed process would correspond to the operator $\frac 12 \Delta + \frac{h'(x)}{h(x)} \frac{d}{dx}$. Whether the usual probabilistic  factor $\frac 12$ in front of the Laplacian is a real difference between the two operators or whether the two operators are actually equal under proper scaling, we are not able to determine due to considerable differences in the presentations of the models in \cite{Jer,Zav06} and in our paper. Either way, we consider it remarkable that significantly different families of processes (reflected Brownian motions and Knudsen random walks) have  limits that are so closely related.

There has been recently interest in billiards in fractal domains. The authors
of \cite{LN1,LN2} take the ``classical'' approach in which the reflection
is specular, that is, the angle of reflection is the same as the angle of
incidence. This idea can be applied in ``prefractals'' approximating,
for example, the von Koch snowflake, and then one can hope to pass
to the limit, in some sense.
Another approach, based on Lambertian reflections, was taken in \cite{CPSV09,CPSV10}.
It was proved in \cite{ABS13} that Lambertian reflections are the only physically realistic reflections if the distribution of the reflected path does not depend on the location of reflection and the incidence angle. As a prelude to the study of fractal domains, the authors of \cite{BT2,BT16} investigated Lambertian reflections in thin tubes; this  shed a light on the distribution of light rays leaving crevices in fractal domains. The present project may be considered as a continuation of \cite{BT2,BT16} although no invariance principle was proved in those papers.

The present article is focused on two-dimensional domains only, as a result of research results in \cite{BT2,BT16}. It was shown in \cite{BT16} that Knudsen's random walk in a two-dimensional tube has steps with infinite variance but the step distribution is nevertheless in the domain of attraction of the normal law---a rare occurrence in probability literature. The variance of the  steps is finite in dimensions 3 and higher, so less interesting (see \cite{BT2}). Moreover, formulas become cumbersome in higher dimensions. The same remarks explain why we put our process inside a circular tube rather than a straight tube with variable width. In the latter case steps could have infinite variance or, in some cases, the light ray could escape to infinity in one go.

 The invariance principle or, at least, the central limit theorem,  for billiards has received some attention when the reflection is random (see \cite[Theorems 2.1 and 2.2]{CPSV10} or \cite[Theorem 3]{CFZ16} for the case when the reflecting angle is chosen among finitely many) or deterministic when the domain has cusps (see \cite{BCD11}). In those invariance principles the domain is fixed and  time is accelerated.

Interest in stochastic billiards arose when researchers started to investigate deterministic billiards with microscopic irregularities at the boundary (see e.g \cite{ CF12, FY04, F07}). Instead of zooming in on these irregularities to do  deterministic analysis, the idea was to consider irregularities as points of random refections.  It turns out that the Lambertian distribution is the invariant and  ergodic probability measure for such random processes, in an appropriate sense (see e.g. \cite{CF12, DLL13}). 

\medskip

On the technical side, we will use two classical versions of the invariance principle, available in \cite{EthierKurtz}. The main effort will be in verifying the assumptions of those theorems. The ballistic character of our process and the smoothness of the boundary make the calculations harder than in the Brownian case---a situation that seems paradoxical but it is well known in other  contexts.

\subsection{Organization of the paper} 

Sections \ref{ResultAnnulus}-\ref{o22.7}
are devoted to the simplified model, in which the domain is a true annulus, i.e., its two parts of the boundary are concentric circles. This may be helpful to the reader as our general result, presented and proved in Sections \ref{DistortedTyre}-\ref{j4.5}, has a proof that contains many details which obscure the basic strategy.

We would like to point out  Proposition \ref{Distribution}, a result that may have a separate interest. It holds only in the case when the domain is a true annulus.

\section{Reflections in an annulus: model and results}\label{ResultAnnulus}

Given $r>0$ and $\eps\in(0,1)$, let
\begin{equation}
\mathcal{D}(\epsilon , r)=\left\{(x,y)\in\mathbb{R}^2: (r-\eps)^2\leq x^2+y^2\leq r^2\right\}.
\end{equation}
We will use $\calC((x,y),r)$ to denote the circle with center $(x,y)$ and radius $r$. We will refer to $\calCi:=\calC((0,0),r-\eps)$ as the inner boundary of $\mathcal{D}(\epsilon , r)$ and  to $\calCo:=\calC((0,0),r)$ as the outer boundary of $\mathcal{D}(\epsilon , r)$.

We will consider
a ray of light  traveling inside $\mathcal{D}(\epsilon , r)$ and reflecting from the boundary. Its position at time $t\geq 0 $ will be denoted
\begin{align}\label{o13.1}
Q(t) =\bfr(t) \left(\cos\bbet(t),\sin\bbet(t)\right).
\end{align}

We give a label to the following assumption for later reference.
\medskip

{\narrower\narrower
 \noindent
($\mathbf{A}$)
We will assume that the light ray always travels with speed 1. 
Every time the light ray is reflected, the reflection angle is independent from the past trajectory and has the Lambertian distribution, i.e., the reflection angle $\Theta$ with respect to the inner normal vector at the point of reflection has the probability density given by
\begin{equation}\label{LambertianDistrib}
\P(\Theta \in d\theta) =
\frac{1}{2}\cos(\theta)d\theta \qquad \text{  for  } \theta \in(-\pi/2, \pi/2).
\end{equation}
}

It is easy to see that the light ray process is invariant under scaling, i.e., if  the process in  
$\mathcal{D}(\epsilon , r)$ is denoted $\{\bfr(t) \left(\cos\bbet(t),\sin\bbet(t)\right), t\geq 0 \}$ 
then for $c>0$,
\begin{align*}
\{c\bfr(t/c) \left(\cos\bbet(t/c),\sin\bbet(t/c)\right), t\geq 0 \}
\end{align*}
is the analogous process in 
$\mathcal{D}(c\epsilon , cr)$. For this reason, we will assume that the light ray travels inside $\mathcal{D}(\epsilon , 1)$ in Sections \ref{ResultAnnulus}-\ref{o22.7}. Since $\eps>0$ remains the only parameter, we will incorporate it in the notation by writing  $\{\bfr^\eps(t) \left(\cos\bbet^\eps(t),\sin\bbet^\eps(t)\right), t\geq 0 \}$.

We  now  state our main result on reflections in an annulus.

\begin{theorem}\label{MainTHM}
Processes $\left\{\bbet^\eps\left(\frac{\pi }{\eps\log(1/\eps)}t\right), t\geq 0\right\}$
converge in law to  Brownian motion in the Skorokhod topology as $\epsilon$ goes to $0$.

\end{theorem}

The proof will be given at the end of Section \ref{o22.7}.

\section{Reflections in an annulus: proofs}\label{o22.7}

We start with some notation. We will write $\bone_a(b) = 1$ if $a=b$ and  $\bone_a(b) = 0$ otherwise. Similarly, for a set $A$, we will say $\bone_A(b) = 1$ if $b\in A$ and  $\bone_A(b) = 0$ otherwise.

We will define a number of objects needed in the proofs.
We will assume that the light ray is on the boundary of 
$\mathcal{D}(\epsilon , 1)$ at time $t=0$, as it clearly  does not affect
the validity of Theorem \ref{MainTHM}.

 We will encode the $n$-th reflection point as 
\begin{equation}\label{o13.2}
 (1-\bs^\eps_n\epsilon)\left(\cos(\bal^\eps_n),\sin(\bal^\eps_n)\right),
\end{equation}
where $\bs^\eps_n$ can be 0 or 1, and $\bal^\eps_n\in \R$ is chosen for $n\geq 0$ so that $|\bal^\eps_{n+1} - \bal^\eps_n| < \pi$.
By convention, the first reflection  occurs at time $t=0$.

It is clear that $\{(\bal^\eps_n,\bs^\eps_n), n\geq 0\}$ is a time homogeneous discrete time Markov chain.

Since the light ray travels  with speed 1, the time  between the $k$-th and $(k-1)$-st reflections can be calculated as 
\begin{align}\label{o19.1}
\Delta\calT^\eps_{k} := 
\left| (1-\bs^\eps_{k}\epsilon)\left(\cos(\bal^\eps_{k}),\sin(\bal^\eps_{k})\right)
-  (1-\bs^\eps_{k-1}\epsilon)\left(\cos(\bal^\eps_{k-1}),\sin(\bal^\eps_{k-1})\right)
\right|.
\end{align}
Set $\mathcal{T}_0^\epsilon=0$ and for $n\geq 1$,
\begin{align}\label{o22.8}
\mathcal{T}_n^\epsilon=\sum_{k=1}^n\Delta\calT_k^\epsilon.
\end{align}
Given $t>0$ and $\eps\in(0,1)$, let
\begin{equation}\label{TimeInverse}
N^\epsilon(t)=\inf\left\{n\geq 0 \; :\; \mathcal{T}_{n+1}^\epsilon> t\right\}=\sup\left\{n\geq 0 \; :\;  \mathcal{T}_{n}^\epsilon\leq t \right\}.
\end{equation}
Then $N^\epsilon(t)$ is  the number of reflections made by the light ray before time $t$, while $\mathcal{T}_n^\epsilon$ represents the time of the $n$-th reflection. With this notation, using \eqref{o13.1}, we can rewrite \eqref{o13.2} as
\begin{equation}\label{o13.3}
Q(\mathcal{T}_n^\epsilon)=
 (1-\bs^\eps_n\epsilon)\left(\cos(\bal^\eps_n),\sin(\bal^\eps_n)\right).
\end{equation}

We will derive formulas linking the angle of reflection $\Theta$ with the increment of angle $\bbet$ between reflections. 
Since $\{(\bal^\eps_n,\bs^\eps_n), n\geq 0\}$ is a time homogeneous Markov chain, it will suffice to analyze $\bal^\eps_1-\bal^\eps_0$.
By rotation invariance of the process, we may and will assume without loss of generality that $\bal^\eps_0=\pi/2$.
Set
 \begin{equation}\label{o15.1}
a= a(\theta)=\tan \left(\frac{\pi}{2}-\theta\right)=1/\tan(\theta) = \cot(\theta).
\end{equation}

 Suppose that $\bs^\eps_0 = 1$, i.e., the light ray starts at the inner circle. 
Then the next reflection must be on the outer circle. If $\Theta =0$ then  $\bal^\eps_1-\bal^\eps_0=0$.

We will denote the coordinates of the second reflection point $(x,y)= (x(\Theta), y(\Theta))$, i.e.,
\begin{equation}
(x(\Theta), y(\Theta))= (1-\bs^\eps_1\epsilon)\left(\cos(\bal^\eps_1),\sin(\bal^\eps_1)\right).
\end{equation}
Then
\begin{equation}\label{NPfromInner}
\bal^\eps_1-\bal^\eps_0= 
 \arctan\left(\frac {x(\Theta)}{y(\Theta)}\right)
= 
 \arctan\left(\frac {x(\Theta)}{a(\Theta)x(\Theta) + 1-\eps}\right).
\end{equation}
If $\Theta \neq 0$, then $x=x(\Theta)$ is the solution of 
\begin{equation}
x^2+y^2=
x^2+(ax+1-\eps)^2=1
\end{equation}
such that $xa > 0$. Elementary computations yield
\begin{equation}\label{xreps}
x= 
\frac{-a(1-\eps)+\sgn(a)\sqrt{a^2-\eps^2+2 \eps}}{1+a^2}.
\end{equation}
For $a>0$, we obtain the following formula using \eqref{o15.1},
\begin{align}\label{o15.4}
x(\theta)&= 
\frac{-a(\theta)(1-\eps)+\sqrt{a(\theta)^2-\eps^2+2 \eps}}{1+a(\theta)^2}
= \frac{-\cot(\theta)(1-\eps)+\sqrt{\cot(\theta)^2-\eps^2+2 \eps}}{1+\cot(\theta)^2}\nonumber\\
&= \sin(\theta)\cos(\theta)
\left(-1+\eps + \sqrt{1+(2-\eps)\eps\tan^2(\theta)}\right).
\end{align}

Suppose that $\bs^\eps_0 = 0$, i.e., the light ray starts at the outer boundary. 
Then the next reflection may occur at the outer or inner boundary.

\begin{lemma}\label{critRemain} 
If $\bs^\eps_0 = 0$ then $\bs^\eps_1 = 0$ if and only if 
$$a(\Theta)^2 < \frac{2 \eps-\eps^2}{(1-\eps)^2}.$$
\end{lemma}
\begin{proof} The light ray hits the inner boundary if and only if there is a solution to 
\begin{equation}
x^2+(ax+1)^2=(1-\eps)^2.
\end{equation}
This equation has a solution if and only if 
\begin{equation}\label{discrim}
a^2 - 2 \eps - 2 a^2 \eps + \eps^2 + a^2 \eps^2 \geq 0,
\end{equation}
i.e., if and only if
\begin{equation}
a^2\geq  \frac{2 \eps-\eps^2}{(1-\eps)^2}.
\end{equation}
\end{proof}

\begin{lemma}\label{remain} 
We have $\P(\bs^\eps_1 = 0 \mid \bs^\eps_0 = 0)=\eps$.
\end{lemma}
\begin{proof}
Set $\gamma(\epsilon )=\arctan\left(\sqrt{(2 \eps-\eps^2)/((1-\eps)^2})\right)$. Then, 
by Lemma \ref{critRemain} and using the fact that $\cos(\arctan(x))=(1+x^2)^{-1/2}$, we have
\begin{align*}
\P(\bs^\eps_1 = 0 \mid \bs^\eps_0 = 0)&=\mathbb{P}\left(\vert a(\Theta)\vert < \sqrt{\frac{2 \eps-\eps^2}{(1-\eps)^2}}\right)\notag
= \mathbb{P}\left(\vert \Theta\vert \in \left[\frac{\pi}{2}- \gamma(\epsilon), \frac{\pi}{2}\right]\right)\notag\\
&= 2\int_{\frac{\pi}{2}-\gamma(\epsilon )}^{\frac{\pi}{2}} \frac{1}{2}\cos(\theta)d\theta
= 1-\cos \gamma(\epsilon )
= 1-\left(1+\frac{2 \eps-\eps^2}{(1-\eps)^2}\right)^{-1/2}= \eps.
\end{align*}
\end{proof}

The following representation of the process 
$\{(\bal^\eps_n,\bs^\eps_n), n\geq 0\}$ will be useful. 

\begin{definition}\label{o18.1}

Let $T^\eps_n, n\geq 1$, be i.i.d.~random variables with the distribution of 
$\bal^\eps_1 - \bal^\eps_0$ conditioned on $\{\bs^\eps_0 = 1\}$, i.e., on the event that the light ray starts from the inner boundary.

Let $R^\eps_n, n\geq 1$, be i.i.d.~random variables with the distribution of 
$\bal^\eps_1 - \bal^\eps_0$ conditioned on $\{\bs^\eps_0 = 0, \bs^\eps_1= 1\}$, i.e., on the event that the light ray starts from the outer boundary and the next reflection is on the inner boundary.

Let $S^\eps_n, n\geq 1$, be i.i.d.~random variables with the distribution of 
$\bal^\eps_1 - \bal^\eps_0$ conditioned on $\{\bs^\eps_0 = 0, \bs^\eps_1= 0\}$, i.e., on the event that the light ray starts from the outer boundary and the next reflection is also on the outer boundary.

Let $\Lambda^\eps_n, n\geq 1$, be i.i.d.~random variables (``Bernoulli sequence'') with the distribution given by
$\P(\Lambda^\eps_n = 1)=1- \P(\Lambda^\eps_n = 0)=\eps$.

We assume that all  random variables 
defined above, for all $n\geq 1$ and $\eps\in(0,1)$, are jointly independent.
\end{definition}
  
We can represent the process $\{(\bal^\eps_n,\bs^\eps_n), n\geq 0\}$ as follows. For $n\geq 0$,
\begin{equation}\label{InducedMC0}
\left\{ 
\begin{array}{l}
\bs_{n+1}^\epsilon = (1-\bs_n^\epsilon)(1-\Lambda_{n+1}^\epsilon),\\
\bal_{n+1}^\epsilon = \bal_n^\epsilon + T_{n+1}^\epsilon \bs_n^\epsilon
+(1-\bs_n^\epsilon)\left(\Lambda_{n+1}^\epsilon S_{n+1}^\epsilon+(1-\Lambda_{n+1}^\epsilon)R_{n+1}^\epsilon\right).
\end{array}
\right.
\end{equation}

We record the following property of random variables $T^\eps_n$ and $ R^\eps_n$ because it is useful in our arguments but we also find the property interesting on its own.

\begin{proposition}\label{Distribution} 
Random variables  $T^\epsilon_n$ and $ R^\epsilon_n$
have the same distribution.
\end{proposition}

\begin{proof}
Recall the notation from \eqref{o13.3}. The following claims follow from \cite[Th. 2.1]{CPSV09}.
The discrete Markov chain $\{Q(\mathcal{T}_n^\epsilon), n\geq 0\}$ representing consecutive reflection locations
has a stationary distribution. (The stationary distribution is uniform on the boundary of 
$\mathcal{D}(\epsilon , 1)$ but this is not relevant in this proof.) The Markov chain is symmetric (see the first displayed formula on page 507 of \cite{CPSV09}) and its time reversal has the same distribution as the process itself. Consider any $-\infty<b_1< b_2 <\infty$ and let $N_+(b_1,b_2, t)$ be the number of $n$ such that $\mathcal{T}_{n+1}^\epsilon \leq t$, $Q(\mathcal{T}_n^\epsilon)\in \calCi$, $Q(\mathcal{T}_{n+1}^\epsilon)\in \calCo$ and $\bal^\eps_{n+1} - \bal^\eps_n \in (b_1, b_2)$. By the ergodic theorem, $\lim_{t\to \infty} N_+(b_1,b_2, t)/t \to \ell_+\in[0,\infty)$.

We will apply the same argument to the ``reversed events.''
Let $N_-(b_1,b_2, t)$ be the number of $n$ such that $\mathcal{T}_{n+1}^\epsilon \leq t$, $Q(\mathcal{T}_n^\epsilon)\in \calCo$, $Q(\mathcal{T}_{n+1}^\epsilon)\in \calCi$ and $\bal^\eps_{n+1} - \bal^\eps_n \in (-b_2, -b_1)$. By the ergodic theorem, $\lim_{t\to \infty} N_-(b_1,b_2, t)/t \to \ell_-\in[0,\infty)$.
Since the time reversed process has the same distribution as the original one, $\ell_- = \ell_+$. 

The above observations, the symmetry of the reflection angle and the rotation invariance of the model easily imply the lemma.
\end{proof}

The proof of Proposition \ref{Distribution} is based on the symmetry of the process of Lambertian reflections, i.e., the fact that the time reversed process has the same distribution as the original one. This symmetry is not obvious so we will present a physical heuristic argument which  makes this symmetry plausible.
 It has been proved in \cite{ABS13} that (random) Lambertian reflections can be approximated by (deterministic) specular reflections from a collection of finite number of mirrors (a specular reflection occurs when the angle of reflection is equal to the angle of incidence). Time reversibility of the classical optics implies  time reversibility of the process of Lambertian reflections. 

Proposition \ref{Distribution} allows us to rewrite \eqref{InducedMC0} as follows
\begin{equation}\label{InducedMC}
\left\{ 
\begin{array}{l}
\bs_{n+1}^\epsilon = (1-\bs_n^\epsilon)(1-\Lambda_{n+1}^\epsilon),\\
\bal_{n+1}^\epsilon = \bal_n^\epsilon + T_{n+1}^\epsilon \bs_n^\epsilon
+(1-\bs_n^\epsilon)\left(\Lambda_{n+1}^\epsilon S_{n+1}^\epsilon+(1-\Lambda_{n+1}^\epsilon)T_{n+1}^\epsilon\right).
\end{array}
\right.
\end{equation}

Since the evolution of $\{\bs^\epsilon_n, n\geq 0\}$ does not depend on  $\{\bal^\epsilon_n, n\geq 0\}$, it is a Markov chain in its own right. The chain $\{\bs^\epsilon_n, n\geq 0\}$ is  irreducible and aperiodic because the transition from 0 to 0 is possible. The unique invariant probability measure $\mu^\epsilon$ is given by 
\begin{equation*}
\mu^\epsilon(0) = \frac{1}{2-\epsilon},\qquad
 \mu^\epsilon(1)=\frac{1-\eps}{2-\epsilon}.
\end{equation*} 
From now on we will assume that $\bs^\eps_0$ (and, therefore, $\bs^\eps_n$ for all $n\geq 0$) is distributed according to $\mu^\eps$. It is easy to see that this assumption does not affect the validity of our main results.

 \begin{lemma}\label{Support}
 
 Set $b_\epsilon= \arctan\left(\sqrt{2\epsilon -\epsilon^2}\right)$. Then the support of the distribution of $T^\epsilon_n$ is $[-b_\epsilon, b_\epsilon]$, while the support of the distribution of $S^\eps_n$ is $[-2b_\epsilon, 2b_\epsilon]$.
 \end{lemma}
 \begin{proof}
Suppose that $\bs^\eps_0 = 1$ and $\bal^\eps_0 = \pi/2$, i.e., the light ray starts from the top point on the inner boundary. 
Recall the representation of the jumps and the notation introduced in \eqref{NPfromInner}. The absolute value of the angular component of the jump $|\bal^\eps_1-\bal^\eps_0|$ is  maximized when $\vert a(\Theta)\vert$ is minimized; in other words, if $\Theta=\pm \pi/2$. 
It is easy to check that when $\Theta=\pm \pi/2$ then $|\bal^\eps_1-\bal^\eps_0| = b_\eps$. This proves our claim about the support of the distribution of $T^\eps_n$. 

Next suppose that $\bs^\eps_0 = 0$ and $\bal^\eps_0 = \pi/2$, i.e., the light ray starts from the top point on the outer boundary. 
Assume that
$|\bal^\eps_1-\bal^\eps_0|$ corresponds to a jump from the outer boundary to outer boundary. Then this quantity 
is maximal when the light ray  is almost tangent to the inner boundary. 
Simple geometry shows that the length of such a light ray segment is bounded by twice the maximum length of a light ray starting from the inner boundary and ending at the outer boundary. By the first part of the proof, 
$|\bal^\eps_1-\bal^\eps_0|$ corresponding to a jump from the outer boundary to outer boundary
is bounded by $2 b_\eps$. This proves the second claim of the lemma.
 \end{proof}

 \begin{lemma}\label{variance2}
 
 We have
 \begin{equation}\label{d20.11}
 \lim_{\epsilon\rightarrow 0}\frac{\E  \big(\left( T^\epsilon_1 \right)^2\big)}{(1/2)\epsilon^2 \log (1/\epsilon)}=1.
 \end{equation}
 \end{lemma}

 \begin{proof}

We will use  formula \eqref{o15.4}, i.e.,
\begin{align}\label{o15.5}
x(\theta)= \sin(\theta)\cos(\theta)
\left(-1+\eps + \sqrt{1+(2-\eps)\eps\tan^2(\theta)}\right).
\end{align}

 We will use the notation introduced in \eqref{o15.1}-\eqref{NPfromInner}. Hence we can and will identify $T^\epsilon_1$ with a function of $\Theta$, i.e.,
$T^\epsilon_1(\Theta)= \arctan\left(x(\Theta)/y(\Theta)\right)$.
Assume that $\bs^\eps_0 = 1$ and $\bal^\eps_0 = \pi/2$.
Then $1-\eps \leq y(\Theta) \leq 1$. Recall from Lemma \ref{Support} that 
$|T^\epsilon_1(\Theta)| \leq \arctan\left(\sqrt{2\epsilon -\epsilon^2}\right)$. These observations  imply that 
\begin{align}\label{o17.1}
  \lim_{\epsilon\rightarrow 0}\frac{\E  \left(\left( T^\epsilon_1 \right)^2\right)}{(1/2)\epsilon^2 \log (1/\epsilon)}
= \lim_{\epsilon\rightarrow 0}\frac{\E \left( \arctan^2\left(x(\Theta)/y(\Theta)\right)\right)}
{(1/2)\epsilon^2 \log (1/\epsilon)}
= \lim_{\epsilon\rightarrow 0}\frac{\E \left( x(\Theta)^2\right)}
{(1/2)\epsilon^2 \log (1/\epsilon)},
\end{align}
assuming that at least one of these limits exists.
It follows from \eqref{o15.5} that
\begin{align}
\E \left( x(\Theta)^2\right)
&= 2 \int_0^{\pi/2}
x(\theta)^2 \frac 1 2 \cos (\theta)  d\theta\nonumber\\
&=  \int_0^{\pi/2}
\sin^2(\theta)\cos^2(\theta)
\left(-1+\eps + \sqrt{1+(2-\eps)\eps\tan^2(\theta)}\right)^2 \cos (\theta)  d\theta\nonumber
\\
&=  \int_0^{\pi/2}
\sin^2(\theta)\cos^3(\theta)
\left(-1+\eps + \sqrt{1+(2-\eps)\eps\tan^2(\theta)}\right)^2  d\theta.
\label{o17.2}
\end{align}

Set 
\begin{align}\label{o17.7}
\theta_0&=\theta_0(\epsilon) = \frac{\pi}{2}
-\arctan\left(2\eps^{1/2}\right).
\end{align}
Then for $\eps\in(0,1)$ and $\theta\in (\theta_0,\frac{\pi}{2})$,
\begin{align*}
\sqrt{1+(2-\eps)\eps\tan^2(\theta)} \leq
\sqrt{2\cdot 1} + \sqrt{2(2-\eps)\eps\tan^2(\theta)}
\leq 2+ 2 \sqrt{\eps}\tan(\theta).
\end{align*}
This implies that, for $\eps\in(0,1)$ and $\theta\in (\theta_0,\frac{\pi}{2})$,
\begin{align*}
&\left(-1+\eps + \sqrt{1+(2-\eps)\eps\tan^2(\theta)}\right)^2
\leq (1-\eps)^2 +  \left( \sqrt{1+(2-\eps)\eps\tan^2(\theta)}\right)^2\\
&\qquad\leq 1 +  \left( 2+ 2 \sqrt{\eps}\tan(\theta)\right)^2
\leq 1 + \left(2 \cdot 2^2+ 2\cdot\left(2 \sqrt{\eps}\tan(\theta)\right)^2
\right)\\
&\qquad= 9 + 8 \eps \tan^2(\theta).
\end{align*}
Hence,
\begin{align}\label{o17.3}
& \int_{\theta_0}^{\pi/2}
\sin^2(\theta)\cos^3(\theta)
\left(-1+\eps + \sqrt{1+(2-\eps)\eps\tan^2(\theta)}\right)^2  d\theta\\
&\leq
\int_{\theta_0}^{\pi/2}
\sin^2(\theta)\cos^3(\theta)
\left(9 + 8 \eps \tan^2(\theta)\right)  d\theta\nonumber\\
& = \int_{\theta_0}^{\pi/2}
9\sin^2(\theta)\cos^3(\theta) d\theta
+
\int_{\theta_0}^{\pi/2}
8 \eps\sin^4(\theta)\cos(\theta) d\theta\nonumber\\
&\leq \int_{\theta_0}^{\pi/2}
9\cos^3(\theta) d\theta
+
\int_{\theta_0}^{\pi/2}
8 \eps\cos(\theta) d\theta\nonumber\\
&\leq  \int_{\theta_0}^{\pi/2}
9(\pi/2 - \theta)^3 d\theta
+
\int_{\theta_0}^{\pi/2}
8 \eps(\pi/2 - \theta) d\theta\nonumber\\
& = 9 \cdot \frac 1 4 (\pi/2 - \theta)^4
+ 8 \eps \frac 1 2 (\pi/2 - \theta)^2\nonumber\\
& = 9 \cdot \frac 1 4 \arctan^4\left(2\eps^{1/2}\right)
+ 8 \eps \frac 1 2 \arctan^2\left(2\eps^{1/2}\right)\nonumber\\
&= O(\eps^2).\nonumber
\end{align}

For $\eps\in(0,1)$ and $\theta\in [0,\theta_0 ]$, $\eps\tan^2(\theta) \leq 1/4$, so
\begin{align*}
\sqrt{1+(2-\eps)\eps\tan^2(\theta)} 
&= 1 + \frac 1 2 (2-\eps)\eps\tan^2(\theta)
+ O\left( \left((2-\eps)\eps\tan^2(\theta)\right)^2\right)\\
&= 1 + \eps\tan^2(\theta) -\frac 1 2 \eps^2\tan^2(\theta)
+ O\left(\eps^2\tan^4(\theta)\right).
\end{align*}
It follows that
\begin{align}\label{o17.4}
&\left(-1+\eps + \sqrt{1+(2-\eps)\eps\tan^2(\theta)}\right)^2\\
&= \left(-1+\eps +  1 + \eps\tan^2(\theta) -\frac 1 2 \eps^2\tan^2(\theta)
+ O\left(\eps^2\tan^4(\theta)\right)\right)^2\nonumber\\
&= \left(\eps \frac 1 {\cos^2(\theta)} -\frac 1 2 \eps^2\tan^2(\theta)
+ O\left(\eps^2\tan^4(\theta)\right)\right)^2\nonumber\\
& = \eps^2 \frac 1 {\cos^4(\theta)}
+ \frac 1 4  \eps^4\tan^4(\theta)
+ O\left(\eps^4\tan^8(\theta)\right)\nonumber\\
&\qquad
+ O\left(\eps^3 \frac {\tan^2(\theta)}{\cos^2(\theta)} \right)
+ O\left(\eps^3 \frac {\tan^4(\theta)}{\cos^2(\theta)} \right)
+ O\left(\eps^4 \tan^6(\theta) \right)\nonumber\\
& = \eps^2 \frac 1 {\cos^4(\theta)}
+ \frac 1 4  \eps^4\frac {\sin^4(\theta)}{\cos^4(\theta)}
+ O\left(\eps^4\frac {\sin^8(\theta)}{\cos^8(\theta)}\right)\nonumber\\
&\qquad
+ O\left(\eps^3 \frac {\sin^2(\theta)}{\cos^4(\theta)} \right)
+ O\left(\eps^3 \frac {\sin^4(\theta)}{\cos^6(\theta)} \right)
+ O\left(\eps^4 \frac {\sin^6(\theta)}{\cos^6(\theta)} \right)\nonumber\\
& = \eps^2 \frac 1 {\cos^4(\theta)}
+  O\left( \eps^4\frac {1}{\cos^4(\theta)}\right)
+ O\left(\eps^4\frac {1}{\cos^8(\theta)}\right)\nonumber\\
&\qquad
+ O\left(\eps^3 \frac {1}{\cos^4(\theta)} \right)
+ O\left(\eps^3 \frac {1}{\cos^6(\theta)} \right)
+ O\left(\eps^4 \frac {1}{\cos^6(\theta)} \right)\nonumber\\
& = \eps^2 \frac 1 {\cos^4(\theta)}
+  O\left( \eps^3\frac {1}{\cos^6(\theta)}\right)
+ O\left(\eps^4\frac {1}{\cos^8(\theta)}\right).\nonumber
\end{align}
We have
\begin{align}\label{o17.5}
& \int_0^{\theta_0}
\sin^2(\theta)\cos^3(\theta)
\frac {1}{\cos^6(\theta)} d\theta
\leq \int_0^{\theta_0}
\frac {1}{\cos^3(\theta)} d\theta 
\leq \int_0^{\theta_0}
\frac {1}{(1-2\theta/\pi)^3} d\theta\\
&\qquad = \frac { \pi (\pi - \theta_0) \theta_0}{(\pi - 2\theta_0)^2}
= O( \eps^{-1} ),\nonumber
\end{align}
and
\begin{align}\label{o17.6}
& \int_0^{\theta_0}
\sin^2(\theta)\cos^3(\theta)
\frac {1}{\cos^8(\theta)} d\theta
\leq \int_0^{\theta_0}
\frac {1}{\cos^5(\theta)} d\theta 
\leq \int_0^{\theta_0}
\frac {1}{(1-2\theta/\pi)^5} d\theta\\
&\qquad = 
\frac{1}{8} \pi  \left(\frac{\pi ^4}{(\pi -2 \theta_0)^4}-1\right)
= O( \eps^{-2} ).\nonumber
\end{align}
It follows from \eqref{o17.7}, \eqref{o17.5} and \eqref{o17.6} that
\begin{align} \nonumber 
& \int_0^{\theta_0}
\sin^2(\theta)\cos^3(\theta)
\left(-1+\eps + \sqrt{1+(2-\eps)\eps\tan^2(\theta)}\right)^2  d\theta\\
& = \int_0^{\theta_0}
\sin^2(\theta)\cos^3(\theta)
\left(\eps^2 \frac 1 {\cos^4(\theta)}
+  O\left( \eps^3\frac {1}{\cos^6(\theta)}\right)
+ O\left(\eps^4\frac {1}{\cos^8(\theta)}\right)\right)  d\theta \nonumber \\
& = \eps^2 \int_0^{\theta_0}
\frac{\sin^2(\theta)}{\cos(\theta)}d\theta +  O( \eps^{2}) \label{d17.6}\\
& = \frac{\eps^2}{2} \left(
 \log\left(1+\sin (\theta_0)\right)
-\log(1-\sin(\theta_0))
-\sin(\theta_0)\right) +  O( \eps^{2}) \nonumber \\
& = - \frac{\eps^2}{2} \log\left(1-\cos (\pi/2-\theta_0)\right) +  O( \eps^{2})  \label{d17.7}\\
& = - \frac{\eps^2}{2} \log\left(1-\cos \left(\arctan\left(2\eps^{1/2} \right)\right)\right) +  O( \eps^{2}) \nonumber \\
& = - \frac{\eps^2}{2} \log\left(2\eps\right) +  O( \eps^{2}) \nonumber 
\\
& =  \frac{\eps^2}{2} |\log \eps| +  O( \eps^{2}).  \nonumber 
\end{align}
This estimate and \eqref{o17.3} imply that
\begin{align*}
& \int_0^{\pi/2}
\sin^2(\theta)\cos^3(\theta)
\left(-1+\eps + \sqrt{1+(2-\eps)\eps\tan^2(\theta)}\right)^2  d\theta
=  \frac{\eps^2}{2} |\log \eps| +  O( \eps^{2}).
\end{align*}
The lemma follows from this, \eqref{o17.1} and \eqref{o17.2}.
\end{proof}

\begin{lemma}\label{EstimParam}
Recall the definition of $\Delta\calT^\eps_k$ stated in \eqref{o19.1}.
For every $k\geq 1$,
\begin{align}\label{o20.4}
\lim_{\epsilon\rightarrow 0}\frac 1 {\epsilon}
\mathbb{E}\left(\Delta\calT_k^\epsilon \mid \bs_{k-1}^\epsilon = 1\right)
&=
\frac{\pi}{2},\\
\lim_{\epsilon\rightarrow 0}\frac 1 {\epsilon}
\mathbb{E}\left(\Delta\calT_k^\epsilon \mid \bs_{k-1}^\epsilon = 0\right)
&=
\frac{\pi}{2}.\label{o20.6}
\end{align}
\end{lemma} 
\begin{proof}

It will suffice to prove the lemma for $k=1$.
By rotation invariance, we can and will assume that $\bal^\eps_0 = \pi/2$.
Then \eqref{o19.1}, \eqref{InducedMC} and Definition \eqref{o18.1} yield
\begin{align}\label{o21.1}
\mathbb{E}\left(\Delta\calT_1^\epsilon \mid \bs_0^\epsilon = 1\right)
&= \E
\sqrt{\sin^2(T^\eps_1) + (1-\eps - \cos(T^\eps_1))^2}\\
&= \E
\sqrt{2(1-\eps)(1- \cos(T^\eps_1))+\eps^2} .\nonumber
\end{align}
Let
\begin{align*}
G(\epsilon)=\frac{1}{\epsilon}\mathbb{E}
\left(\sqrt{2(1-\eps)(1- \cos(T^\eps_1))+\eps^2}
\right).
\end{align*}
Since $\vert T^\epsilon_1\vert \leq \arctan\big(\sqrt{2\epsilon+\epsilon^2}\big)$ by Lemma \ref{Support}, the Taylor expansion for the cosine function  at $0$ yields
\begin{equation}\label{o21.2}
1-\cos(T^\epsilon_1)=\frac{1}{2}(1+O(\epsilon))\big(T^\epsilon_1\big)^2.
\end{equation}  
Therefore, using notation from \eqref{NPfromInner},
\begin{align}\label{d25.1}
G(\epsilon)&=\mathbb{E}\left(\sqrt{1+(1+\epsilon)(1+O(\epsilon))(T^\eps_1/\eps)^2}\right)\\
&=\mathbb{E}\left(\sqrt{1+(1+\epsilon)(1+O(\epsilon))
\left(\frac 1 \eps \arctan(x(\Theta)/y(\Theta))\right)^2}
\mid \bs^\eps_0 = 1\right).\notag
\end{align}

We will estimate $\frac 1 \eps \arctan(x(\Theta)/y(\Theta))$.
The following geometric interpretation of the quantity $\frac 1 \eps \arctan(x(\Theta)/y(\Theta))$ follows from  \eqref{NPfromInner}, rescaling (enlarging) the annulus $\calD(\eps,1)$ by the factor of $1/\eps$, and then shifting it down by $1/\eps$ so that its outer boundary passes through the origin. 
Consider the half-line $L$ starting at $(0,-1)$ at an angle $\theta\in[0,\pi/2)$ with the vertical line. Let $A_1(\eps)$ be the intersection point of   $L$ with the circle $\calC((0, - 1/\eps), 1/\eps)$ (i.e., the outer boundary of the transformed domain) and let $A_2$ be the intersection point of   $L$ with the horizontal axis. Then 
$\frac 1 \eps \arctan(x(\Theta)/y(\Theta))$ is the angle between the vertical line and the line passing through points $A_1(\eps)$ and $(0, - 1/\eps)$. 
Let $\alpha(\eps)$ be the angle  between the vertical line and the 
line passing through points $A_2$ and $(0, - 1/\eps)$. 
For every fixed $\theta\in[0,\pi/2)$, $A_1(\eps) \to A_2$ as $\eps\to 0$. 
This implies that 
\begin{align}\label{o12.1}
\lim_{\eps \downarrow 0} \frac{\arctan(x(\Theta)/y(\Theta))}{\epsilon}
= \lim_{\eps \downarrow 0} \frac{\alpha(\eps)}{\epsilon}
= \dist(A_2, (0,0)) = \tan \theta.
\end{align}
Moreover, we have $\arctan(x(\Theta)/y(\Theta)) \leq \alpha(\eps)$ and $\alpha(\eps) \leq \tan \alpha(\eps) = \dist(A_2, (0,0)) / (1/\eps)$ so for all $\eps >0$ and $\theta\in[0,\pi/2)$,
\begin{align}\label{o12.2}
 \frac{\arctan(x(\Theta)/y(\Theta))}{\epsilon}
\leq \frac{\alpha(\eps)}{\epsilon}
\leq \dist(A_2, (0,0)) = \tan \theta.
\end{align}
A similar analysis applies to $\theta\in (-\pi/2, 0]$.
By the dominated convergence theorem and \eqref{o12.1}-\eqref{o12.2},
\begin{align}\label{o20.7}
\lim_{\epsilon\rightarrow 0}\frac 1 {\epsilon}
&\mathbb{E}\left(\Delta\calT_k^\epsilon \mid \bs_{k-1}^\epsilon = 1\right)
=
\lim_{\epsilon\rightarrow 0}
\frac{1}{\epsilon} \E\left( 
\sqrt{2(1-\eps)(1- \cos(T^\eps_1))+\eps^2}\right)\nonumber\\
&=\lim_{\epsilon\rightarrow 0}G(\epsilon)=\mathbb{E}\left(\sqrt{1+\tan^2(\Theta)}\right)
= \mathbb{E}\left(\frac{1}{\cos(\Theta)}\right)=\frac{\pi}{2}.
\end{align}
This proves  \eqref{o20.4}.

By \eqref{o19.1}, \eqref{InducedMC} and Definition \ref{o18.1},
\begin{align}\label{o20.9}
&\mathbb{E}\left(\Delta\calT_1^\epsilon \mid \bs_0^\epsilon = 0\right)\\
&= \E \left(
\Lambda_{1}^\epsilon
\sqrt{2(1-\cos(S^\eps_1))} +(1-\Lambda_{1}^\epsilon)
\sqrt{2(1-\eps)(1- \cos(T^\eps_1))+\eps^2}\right).\nonumber
\end{align}

By Lemma \ref{remain} and Lemma \ref{Support}, 
\begin{align}\label{o20.1}
&\lim_{\epsilon\rightarrow 0}\frac{1}{\epsilon}
\E \left|\Lambda_{1}^\epsilon
\sqrt{2(1-\cos(S^\eps_1))}\right|
= \lim_{\epsilon\rightarrow 0}\frac{1}{\epsilon}
\eps \E \left|\sqrt{2(1-\cos(S^\eps_1))}\right|
= 0,
\end{align}
and
\begin{align}\label{o20.2}
&\lim_{\epsilon\rightarrow 0}\frac{1}{\epsilon}
\E \left|-\Lambda_{1}^\epsilon
\sqrt{2(1-\eps)(1- \cos(T^\eps_1))+\eps^2}\right|\\
&= \lim_{\epsilon\rightarrow 0}\frac{1}{\epsilon}
\eps \E \left(\sqrt{2(1-\eps)(1- \cos(T^\eps_1))+\eps^2}\right)
= 0.\nonumber
\end{align}
By  \eqref{o20.7}, 
\begin{align}\label{o20.8}
\lim_{\epsilon\rightarrow 0}
\frac{1}{\epsilon} \E\left( 
\sqrt{2(1-\eps)(1- \cos(T^\eps_1))+\eps^2}\right)=\frac{\pi}{2}.
\end{align}
The combination of \eqref{o20.9}, \eqref{o20.1}, \eqref{o20.2} 
and \eqref{o20.8} implies \eqref{o20.6}.
\end{proof}

\begin{lemma}\label{EstimNumberBounce} 
For $t\geq 0$, 
\begin{align}\label{n9.1}
\mathbb{E}\left(N^\epsilon (t)\right)\leq 2(t+2\eps)/\epsilon.
\end{align}
\end{lemma}
\begin{proof}
The light ray travels at speed 1 so it takes at least $\eps$ units of time 
between any two consecutive reflections that don't take place on the same piece of the boundary. Thus $n$ crossings from the inner to the outer boundary and $n$ crossings from the outer to the inner boundary must take at least $2n\epsilon$ units of time. 
Let $U(n)$ be the total number of reflections (including consecutive reflections from the outer boundary) that have occurred by the time when 
$n$ crossings from the inner to the outer boundary and $n$ crossings from the outer to the inner boundary have happened. Then $N^\eps(t) \leq U(\lceil t/(2\epsilon)\rceil)$.
We can represent $U(n)$ as
\begin{equation*}
  U(n)=n +\sum_{k=1}^n X^\epsilon_{k},
\end{equation*}
where $X^\eps_k$ are i.i.d. random variables with the geometric distribution (taking values $1,2, \dots$) with parameter $1-\eps$ (see Lemma \ref{remain}).
Therefore, for $\eps < 1/2$,
\begin{align*}
&\mathbb{E}\left(N^\epsilon(t)\right)
\leq \mathbb{E}\left(U\left(\lceil t/(2\epsilon)\rceil
\right)\right)
=  \lceil t/(2\epsilon)\rceil +  \lceil t/(2\epsilon)\rceil \frac 1 {1-\eps}
= \lceil t/(2\epsilon)\rceil \frac {2-\eps} {1-\eps}
\\
&\leq 
\left( \frac t{2\eps} +1\right) \frac {2-\eps} {1-\eps}
\leq 2(t+2\eps)/\epsilon.
\end{align*}
\end{proof}

\begin{lemma}\label{o22.21}
Processes $\left\{
\frac{\pi}{2} \eps^2\log(1/\eps)
N^\epsilon \left(\frac{t}{(1/2)\eps\log(1/\eps)}\right)-t, t\geq 0\right\}$ converge in probability   toward 0 in the uniform topology on compact sets when $\eps\rightarrow 0$.
\end{lemma}

\begin{proof}

Computations similar to those in \eqref{o21.1} and \eqref{o21.2} yield
\begin{align*}
\mathbb{E}\left((\Delta\calT_1^\epsilon)^2 \mid \bs_0^\epsilon = 1\right)
&= \E\left(\sin^2(T^\eps_1) + (1-\eps - \cos(T^\eps_1))^2\right)\\
&= \E \left(2(1-\eps)(1- \cos(T^\eps_1))+\eps^2\right)\\
&= \E \left((1+\epsilon)(1+O(\epsilon))(T^\eps_1)^2+\eps^2\right) .
\end{align*}
This and Lemma \ref{variance2} imply that for small $\eps>0$,
\begin{align}\label{o21.6}
\mathbb{E}\left((\Delta\calT_1^\epsilon)^2 \mid \bs_0^\epsilon = 1\right)
< \epsilon^2\log (1/\epsilon)  .
\end{align}

By \eqref{o20.9}, 
\begin{align}\label{o21.3}
\mathbb{E}&\left((\Delta\calT_1^\epsilon)^2 \mid \bs_0^\epsilon = 0\right)
 \\
&= \E \left( \left(
\Lambda_{1}^\epsilon
\sqrt{2(1-\cos(S^\eps_1))} +(1-\Lambda_{1}^\epsilon)
\sqrt{2(1-\eps)(1- \cos(R^\eps_1))+\eps^2}\right)^2\right)
\nonumber \\
& \leq 
2 \E \left( \left(
\Lambda_{1}^\epsilon \right)^2
\left(2(1-\cos(S^\eps_1))\right) \right)\nonumber 
+2 \E \left((1-\Lambda_{1}^\epsilon)^2
\left(2(1-\eps)(1- \cos(R^\eps_1))+\eps^2\right)\right).
\end{align}

By Lemma \ref{variance2} and Proposition \ref{Distribution}, for small $\eps>0$,
\begin{align}\label{o21.4}
2 \E \left((1-\Lambda_{1}^\epsilon)^2
\left(2(1-\eps)(1- \cos(R^\eps_1))+\eps^2\right)\right)
<  2\epsilon^2\log (1/\epsilon)  .
\end{align}
It follows from the definition of $\Lambda^\eps_1$ and Lemma \ref{Support} that, for small $\eps>0$,
\begin{align*}
2 \E \left( \left(
\Lambda_{1}^\epsilon \right)^2
\left(2(1-\cos(S^\eps_1))\right) \right)
\leq 5\eps^2.
\end{align*}
This, \eqref{o21.3} and \eqref{o21.4} imply that, for small $\eps>0$,
\begin{align}\label{o21.5}
\mathbb{E}\left((\Delta\calT_1^\epsilon)^2 \mid \bs_0^\epsilon = 0\right)
\leq 2\epsilon^2\log (1/\epsilon)  .
\end{align}

Recall definition \eqref{o22.8} and
set 
\begin{align*}
M_n^\epsilon = 
\frac 12 \eps\log(1/\eps)
\left(\calT^\eps_n  -\sum_{k=1}^n\mathbb{E}\big(\Delta\calT_k^\epsilon \mid \mathcal{F}_{k-1}^\epsilon\big)\right).
\end{align*}
Then $(M_n^\epsilon)_{n\geq 0}$ is a martingale starting at $0$ and its quadratic variation is 
$$\langle M^\epsilon\rangle_n =
\frac14 \eps^2\log^2(1/\eps) \sum_{k=1}^n\var (\Delta\calT_k^\epsilon\mid \mathcal{F}_{k-1}^\epsilon).$$
From \eqref{o21.6} and \eqref{o21.5}, we obtain for small $\eps>0$,
\begin{align}\label{o22.1}
\langle M^\epsilon\rangle_n \leq  \frac14 \eps^2\log^2(1/\eps) \sum_{k=1}^n\mathbb{E}(\big(\Delta\calT_k^\epsilon\big)^2\mid \mathcal{F}_{k-1}^\epsilon\big)
\leq \frac 12 n \epsilon^4\log^3 (1/\epsilon).
\end{align} 

In this proof, we will use the notation $W(\eps,t) = N^\epsilon\left(\frac{t}{(1/2)\eps\log(1/\eps))}\right)$.
By Lemma \ref{EstimNumberBounce}, $W(\eps,t)$ is a stopping time with a finite expectation so by the optional stopping theorem,  \eqref{n9.1} and \eqref{o22.1},
for small $\eps>0$, 
\begin{align}\label{o22.1b}
\mathbb{E}\left(\left(M_{W(\eps,t)}^\epsilon\right)^2\right)
&=
\E\langle M^\epsilon\rangle_{W(\eps,t)} 
\leq \frac 12  \epsilon^4\log^3 (1/\epsilon)
\E W(\eps,t)\nonumber\\
& \leq 
 \frac 12  \epsilon^4\log^3 (1/\epsilon)
 \frac 2\eps \left(\frac{t}{(1/2)\eps\log(1/\eps)}+2\epsilon\right) \nonumber\\
& =  \epsilon^3\log^3 (1/\epsilon) 
 \left(\frac{t}{(1/2)\eps\log(1/\eps)}+2\epsilon\right).
\end{align}
For a fixed $t$, the right hand side goes to 0 as $\eps\to 0$.

The definition of $N^\eps(t)$ implies that
\begin{equation}\label{o22.2}
\calT^\eps_{W(\eps,t)} 
\leq \frac{t}{(1/2)\eps\log(1/\eps))}\leq \calT^\eps_{W(\eps,t)+1} 
=\calT^\eps_{W(\eps,t)} 
+\Delta\calT_{W(\eps,t)+1}^\epsilon.
\end{equation}
It follows easily from \eqref{o19.1} and Lemma \ref{Support}
that  $\lim_{\epsilon\rightarrow 0}\sup_{k\geq 1} \Delta\calT_k^\epsilon =0$ almost-surely. Hence, a.s.,
\begin{align}\label{o22.4}
\lim_{\eps\to 0} \eps\log(1/\eps)\left|\calT^\eps_{W(\eps,t)} 
- \frac{t}{(1/2)\eps\log(1/\eps))} \right| = 0.
\end{align}

It follows from the definition of $M^\eps_n$ and \eqref{o22.1b} that
\begin{align}\label{o22.3}
\lim_{\eps\to 0}\eps\log(1/\eps)
\left\vert \calT^\eps_{W(\eps,t)}  
 - \sum_{k=1}^{W(\eps,t)}
\mathbb{E}\left(\Delta\calT_k^\epsilon \mid \mathcal{F}_{k-1}^\epsilon\right)\right\vert =0,
\end{align}
in probability.

Lemma \ref{EstimParam} implies that, a.s.,
\begin{align}\label{n9.2}
\lim_{\eps\to0} 
\left( \frac 1 \eps  \sup_{k\geq 0}
\left|\mathbb{E}\left(\Delta\calT_k^\epsilon \mid \mathcal{F}_{k-1}^\epsilon\right) -\mathbb{E}(\Delta\calT_k^\eps)
\right|\right) = 0.
\end{align}

By  Lemmas \ref{EstimParam} and  \ref{EstimNumberBounce}, and \eqref{n9.2},
for $t>0$,
\begin{align*}
0&\leq \lim_{\eps\to 0}\eps \log(1/\eps)
\E\left\vert\frac{\pi}{2}\eps W(\eps,t)
- \sum_{k=1}^{W(\eps,t)}
\mathbb{E}\left(\Delta\calT_k^\epsilon \mid \mathcal{F}_{k-1}^\epsilon\right)
  \right\vert \\
&\leq \lim_{\eps\to 0}\eps\log(1/\eps)\E\left\vert
\frac{\pi}{2}\eps W(\eps,t) - \sum_{k=1}^{W(\eps,t)}
\mathbb{E}(\Delta\calT_k^\eps)
\right\vert\\
&\qquad + \lim_{\eps\to 0}\eps\log(1/\eps)\E\left\vert \sum_{k=1}^{W(\eps,t)}
\left(
\mathbb{E}\left(\Delta\calT_k^\epsilon \mid \mathcal{F}_{k-1}^\epsilon\right) -\mathbb{E}(\Delta\calT_k^\eps)\right) \right\vert\\
&\leq \lim_{\eps\to 0}\eps\log(1/\eps)\E(W(\eps,t))
\left|\frac{\pi}{2}\eps - \mathbb{E}(\Delta\calT_1^\eps)
 \right\vert\\
&\qquad + \lim_{\eps\to 0}\eps\log(1/\eps)\E(W(\eps,t))
\sup_{k\geq 0}\left|
\mathbb{E}\left(\Delta\calT_k^\epsilon \mid \mathcal{F}_{k-1}^\epsilon\right) -\mathbb{E}(\Delta\calT_k^\eps) \right\vert\\
&\leq \lim_{\eps\to 0}\eps\log(1/\eps) \frac 2 \eps
 \Big(\frac{t}{(1/2)\eps\log(1/\eps))}+2\eps\Big) \eps  
\left|\frac{\pi}{2} - \frac 1\eps \mathbb{E}(\Delta\calT_k^\eps)
 \right\vert\\
&\qquad + \lim_{\eps\to 0}\eps\log(1/\eps) \frac 2 \eps
 \Big(\frac{t}{(1/2)\eps\log(1/\eps))}+2\eps\Big) 
\sup_{k\geq 0}\left|
\mathbb{E}\left(\Delta\calT_k^\epsilon \mid \mathcal{F}_{k-1}^\epsilon\right) -\mathbb{E}(\Delta\calT_k^\eps) \right\vert\\
&\leq \lim_{\eps\to 0}4t
\left|\frac{\pi}{2} - \frac 1\eps \mathbb{E}(\Delta\calT_k^\eps)
 \right\vert + \lim_{\eps\to 0} 4t
\frac 1 \eps 
\sup_{k\geq 0}\left|
\mathbb{E}\left(\Delta\calT_k^\epsilon \mid \mathcal{F}_{k-1}^\epsilon\right) -\mathbb{E}(\Delta\calT_k^\eps) \right\vert\\
&=0.
\end{align*}
This, \eqref{o22.4} and \eqref{o22.3} imply that for any fixed $t\geq 0$,
\begin{align*}
\lim_{\eps\to 0}
\left\vert \frac \pi 2 \eps^2\log(1/\eps) N^\epsilon\left(\frac{t}{(1/2)\eps\log(1/\eps))}\right)  -t \right\vert =
\lim_{\eps\to 0}
\left\vert \frac \pi 2 \eps^2\log(1/\eps) W(\eps,t)   -t \right\vert 
=0,
\end{align*}
in probability. The stronger statement given in the lemma follows from this and the fact that the process $t\to N^\epsilon(t)$ is non-decreasing.
\end{proof}

\begin{proof}[Proof of Theorem \ref{MainTHM}]

First we are going to apply \cite[Theorem 1.4, Chapter 7]{EthierKurtz}
to a time change of  $\bal^\eps_k$. We extend the time parameter for this process from integers to reals by letting $\bal^\eps_t := \bal^\eps_{\lfloor t \rfloor}$ for $t\geq 0$. Next we rescale, i.e., we let
\begin{align*}
\sigma^2_\eps & =  \frac 12 \eps^2 \log(1/\eps),\\
\wt \bal^\eps_t &=  \bal^\eps_{t/\sigma^2_\eps }, \qquad
\text{  for  } t\geq0. 
\end{align*}
We will prove that  processes $\{\wt \bal^\eps_t, t\geq 0\}$ converge weakly to Brownian motion as $\eps \to 0$.

It follows easily from the symmetry of jumps of $\bal^\eps_k$ and from Lemma \ref{Support} that
the process $\{\wt \bal^\eps_t, t\geq 0\}$ is a martingale.

We will use assumption (a) of \cite[Theorem 1.4, Chapter 7]{EthierKurtz}. According to Definition \ref{o18.1}, Proposition \ref{Distribution} and Lemma \ref{Support}, 
$|\bal^\eps_{n+1} - \bal^\eps_n | \leq 2 \arctan\left(\sqrt{2\epsilon -\epsilon^2}\right) $, a.s., for all $n$. Hence, for all $t_0>0$,
\begin{align*}
\lim_{\eps \to 0} \E \left( \sup_{t \leq t_0} 
|\wt\bal^\eps_{t} - \wt\bal^\eps_{t-} | \right) 
\leq \lim_{\eps \to 0} 
2 \arctan\left(\sqrt{2\epsilon -\epsilon^2}\right)
= 0.
\end{align*}
This means that condition (1.14) of \cite[Theorem 1.4, Chapter 7]{EthierKurtz} is satisfied. It remains to show that the quadratic  
variation $\langle \wt\bal^\eps \rangle_t$ of $\wt\bal^\eps$ converges to $t$. More precisely, we have to show that for each $t\geq 0$, $\langle \wt\bal^\eps \rangle_t \to t$ in probability.
We will compute the quadratic variation $\langle \bal^\eps \rangle_n$ of $\bal^\eps_n$ first.

Recall from \eqref{InducedMC} that
\begin{equation*}
\bal_{n+1}^\epsilon = \bal_n^\epsilon + T_{n+1}^\epsilon
\bs_n^\epsilon
+(1-\bs_n^\epsilon)\left(\Lambda_{n+1}^\epsilon S_{n+1}^\epsilon+(1-\Lambda_{n+1}^\epsilon)T_{n+1}^\epsilon\right).
\end{equation*}
We have assumed that  $\{\bs^\epsilon_n, n\geq 0\}$ 
is in the stationary regime, i.e., for all $n\geq 0$,  $\bs^\eps_n$  is distributed according to the stationary distribution $\mu^\eps$, where
\begin{equation}\label{o18.2}
\mu^\epsilon(0) = \frac{1}{2-\epsilon},\qquad
 \mu^\epsilon(1)=\frac{1-\eps}{2-\epsilon}.
\end{equation} 

Let $\mathcal{F}_n^\epsilon=\sigma(\bal_k^\epsilon,\bs_k^\epsilon ,\, k=1,\cdots ,n)$. Then
\begin{align*}
&\mathbb{E}\left(\left(\bal_{n+1}^\epsilon-\bal_n^\epsilon\right)^2 \mid \mathcal{F}_n^\epsilon\right)\\
&=\mathbb{E}\left((T_{n+1}^\epsilon)^2\right) \bs_n^\epsilon
+ \left(\eps\mathbb{E}\left((S_{n+1}^\epsilon)^2\right)
+(1-\eps)\mathbb{E}\left((T_{n+1}^\epsilon)^2\right)\right) (1-\bs_n^\epsilon),
\end{align*}
so
\begin{align}
\langle \bal^\epsilon\rangle_n
&=\sum_{k=0}^n \mathbb{E}\left(\left(\bal_{k+1}^\epsilon-\bal_k^\epsilon\right)^2 \mid \mathcal{F}_k^\epsilon\right)\notag\\
&=\mathbb{E}\left((T_{1}^\epsilon)^2\right)
\sum_{k=0}^n \bs_k^\epsilon
+ \left(\eps\mathbb{E}\left((S_{1}^\epsilon)^2\right)
+(1-\eps)\mathbb{E}\left((T_{1}^\epsilon)^2\right)\right)
\sum_{k=0}^n (1-\bs_k^\epsilon)\notag\\
&=n(1-\eps)\mathbb{E}\left((T_{1}^\epsilon)^2\right)  + \eps \mathbb{E}\left((T_{1}^\epsilon)^2\right) \sum_{k=0}^n \bs_k^\epsilon+\epsilon\mathbb{E}\left((S_{1}^\epsilon)^2\right)\sum_{k=0}^n (1-\bs_k^\epsilon).\label{o18.3}
\end{align}
It follows from Lemmas \ref{Support} and \ref{variance2} that 
\begin{align*}
 &\lim_{\epsilon\rightarrow 0}\frac 1 {\sigma^2_\eps} (1-\eps)
\E  \big(\left( T^\epsilon_1 \right)^2\big)=1,\\
 &\lim_{\epsilon\rightarrow 0}\frac 1 {\sigma^2_\eps}
 \left( \eps\mathbb{E}\left((S_{1}^\epsilon)^2\right)
+\eps\mathbb{E}\left((T_{1}^\epsilon)^2\right) \right)=0.
\end{align*}
This and \eqref{o18.3} imply that, for each $t\geq 0$,
\begin{align*}
\lim_{\eps\to 0} \langle \wt \bal^\epsilon\rangle_{ t} =
\lim_{\eps\to 0} \langle \bal^\epsilon\rangle_{ t/\sigma^2_\eps} =
\lim_{\eps\to 0} \langle \bal^\epsilon\rangle_{\lfloor t/\sigma^2_\eps\rfloor} =t,
\end{align*}
almost-surely. 
This completes the proof that  processes $\{\wt \bal^\eps_t, t\geq 0\}$ converge weakly to  Brownian motion as $\eps \to 0$.

To finish the proof, we need to time change the process $\{\wt \bal^\eps_t, t\geq 0\}$.
More precisely, we note that 
\begin{align}\label{o22.20}
\bbet^\eps(t) = \bal^\eps_{N^\eps(t)}
= \wt \bal^\eps_{\sigma^2_\eps N^\eps(t)}
= \wt \bal^\eps_{ (1/2)\eps^2 \log(1/\eps) N^\eps(t)}
\end{align}
for $t$ at which $N^\eps(t)$ jumps.
We will apply the last formula with $t = \frac{\pi s}{\eps\log(1/\eps)}$. We have
\begin{align}\label{n9.3}
\frac 1 2 \eps^2\log(1/\eps)N^\eps
\left(\frac{\pi s}{\eps\log(1/\eps)}\right)
= s+ \frac{2}{\pi}\Big(\frac{\pi}{2}\eps^2\log(1/\eps)N^\eps
\left(\frac{\pi s}{\eps\log(1/\eps)}\right) - \frac{\pi s}{2}\Big).
\end{align}
The jumps of $\bal^\eps$ are uniformly bounded by a quantity going to 0 when $\eps \to 0$, by Lemma \ref{Support}.
This observation, Lemma \ref{o22.21}, \eqref{o22.20} and \eqref{n9.3}
imply that for a fixed $s\geq 0$,
\begin{align*}
\lim_{\eps\to 0}
\left|\bbet^\eps\left(\frac{\pi s}{\eps\log(1/\eps)}\right)
- \wt \bal^\eps_{s}\right|
=0
\end{align*}
in probability. 
This formula, the uniform bound for the jumps of $\bal^\eps$ and weak convergence of processes $\{\wt \bal^\eps_t, t\geq 0\}$ to Brownian motion imply weak convergence of processes $\{\bbet^\eps(\pi t/(\eps \log(1/\eps))), t\geq 0\}$ to Brownian motion as $\eps \to 0$.
\end{proof}

\section{Reflections in a perturbed annulus: model and results}\label{DistortedTyre}

We will generalize Theorem \ref{MainTHM} to ``perturbed annuli'' 
whose boundaries are smooth curves close to circles. The precise definition follows.

For any function $f: \R \to \R$, $\Vert f \Vert_\infty$ will denote its supremum norm, i.e., 
$\Vert f \Vert_\infty =\sup_{x\in \R} \vert f(x)\vert$.

Let $(f_\epsilon)_{0<\epsilon< 1/2}$ and $(g_\epsilon)_{0<\epsilon< 1/2}$ be families of $2\pi$-periodic $C^3$ functions from $\R$ to $\R$, satisfying the following assumptions. 
\begin{enumerate}
\item[$\bold{H1}$:] For all $\alpha\in [0,2\pi]$, $\epsilon \leq f_\epsilon(\alpha)\leq 2\epsilon$ and $0\leq g_\eps(\alpha)\leq \eps$,
\item[$\bold{H2}$:] $f_\epsilon/\epsilon$ and $g_\epsilon/\epsilon$ converge uniformly to  $f$ and $g$, resp. 
\item[$\bold{H3}$:] $f_\epsilon'/\epsilon$ and $g_\epsilon'/\epsilon$ converge uniformly to $f'$ and $g'$, resp.
\item[$\bold{H4}$:]  $f_\epsilon''/\epsilon$ and $g_\epsilon''/\epsilon$ converge uniformly to $f''$ and $g''$, resp.
\item[$\bold{H5}$:] For some $c<\infty $ and all $\eps\in(0,1/2)$, $\Vert f_\epsilon '''\Vert_\infty < c$ and $\Vert g_\epsilon '''\Vert_\infty < c$.
\end{enumerate}

\begin{remark}
(i) A good example to keep in mind is $f_\epsilon(\alpha)=\epsilon f(\alpha)$ and $g_\eps(\alpha)=\eps g(\alpha)$, where $f:\mathbb{R}\rightarrow [1,2]$ and $g:\mathbb{R}\rightarrow [0,1]$ are
 $2\pi$-periodic $C^3$ function.

(ii) Assuming \bht, if $f_\epsilon'/\epsilon$ and $g_\epsilon'/\epsilon$ converge uniformly then they must converge to $f'$ and $g'$, resp. (see \cite[Thm. 7.17]{Rudin}).

(iii) Assumption \bho{} could have been $c_1\epsilon \leq f_\epsilon(\alpha)\leq c_2\epsilon$ and $c_3\eps\leq g_\eps(\alpha)\leq c_4\eps$, for some constants $0 < c_1< c_2< \infty$ and $0 \leq c_3 < c_4 < \infty$. We gave \bho{} its present form to avoid adding further complexity to the already highly complex notation.
\end{remark}

It will be convenient to use complex notation occasionally. For example, we will write $e^{i\alpha} = \exp(i\alpha) = (\cos\alpha, \sin\alpha)$.

Given $\epsilon\in (0,1/2)$, let $\Gamma_\epsilon^0, \Gamma_\eps^1$ be  closed simple curves parametrized as follows,
\begin{align}
\Gamma_\epsilon^0(\alpha)&=(1+g_\epsilon (\alpha))e^{i\alpha},
\qquad
\Gamma_\epsilon^1(\alpha)=(1-f_\epsilon (\alpha))e^{i\alpha},\label{d2.1}
\end{align} 
for $\alpha\in[0,2\pi)$; the formulas are valid for $\alpha \in \R$ because of the periodicity of $f_\eps$ and $g_\eps$.
Let $\mathcal{U}_\epsilon^j$ denote the bounded connected component of $\mathbb{R}^2\setminus \Gamma_\epsilon^j$ for $j=0,1$.

We  consider a ray of light traveling inside $\mathcal{D}_\epsilon := \overline{\mathcal{U}_\epsilon^0\setminus\mathcal{U}_\epsilon^1}$. Its position at time $t\geq 0$ will be denoted by
\begin{align*}
Q^\eps(t)= \bfr^\eps(t) \exp\left(i \bbet^\eps(t)\right).
\end{align*}
We assume that the trajectory $Q^\eps(t)$ conforms to
($\mathbf{A}$) and \eqref{LambertianDistrib} in Section \ref{ResultAnnulus}.

Our main result on reflections in a perturbed annulus is the following.
\begin{theorem}\label{MainThm3}
Let $h=f+g$.
Processes $\left\{\bbet^\eps\left(\frac{\pi }{\eps\log(1/\eps)}t\right), t\geq 0\right\}$ converge in law to $X$ in the Skorokhod topology as $\epsilon$ goes to $0$, where $X$ solves the stochastic differential equation
\begin{align}\label{j4.1}
dX_t =h'(X_t)dt + \sqrt{h(X_t)}dW_t,
\end{align}
and $W$ is  standard Brownian motion.
\end{theorem}

\section{Reflections in a perturbed annulus: Proofs}\label{j4.5}

We will encode the $n$-th reflection point as 
\begin{equation}\label{o13.2b}
 \wp_\eps(\bal_n^\eps,\bs^\eps_n)\exp\left(i\bal^\eps_n\right)
= \Gamma_\eps^{\bs^\eps_n}\left(\bal_n^\eps\right),
\end{equation}
where $\bs^\eps_n$ can be 0 or 1, $\bal^\eps_n\in \R$ is chosen for $n\geq 0$ so that $|\bal^\eps_{n+1} - \bal^\eps_n| < \pi$ and
\begin{equation}
\wp_\eps(\alpha,s) = 1 +(1-s)g_\eps(\alpha)-s f_\eps(\alpha).
\end{equation}

By convention, the first reflection  occurs at time $t=0$.

We will sometimes write $\bal^\eps(n)$ instead of $\bal^\eps_n$, for typographical convenience.

It is clear that $\{(\bal^\eps_n,\bs^\eps_n), n\geq 0\}$ is a time homogeneous discrete time Markov chain. Since the light ray travels  with speed 1, the time  between the $k$-th and $(k-1)$-st reflections can be calculated as 
\begin{align}\label{o19.1b}
\Delta\calT^\eps_{k} := 
\left| \wp_\eps(\bal_k^\eps,\bs^\eps_k)\exp(i\bal^\eps_{k})
-  \wp_\eps(\bal_{k-1}^\eps,\bs^\eps_{k-1})
\exp(i\bal^\eps_{k-1}) \right|.
\end{align}

Set $\mathcal{T}_0^\epsilon=0$ and for $n\geq 1$,
\begin{align}\label{o22.8b}
\mathcal{T}_n^\epsilon=\sum_{k=1}^n\Delta\calT_k^\epsilon.
\end{align}
Given $t>0$ and $\eps\in(0,1/2)$, let
\begin{equation}\label{TimeInverse2}
N^\epsilon(t)=\inf\left\{n\geq 0 \; :\; \mathcal{T}_{n+1}^\epsilon> t\right\}
=\sup\left\{n\geq 0 \; :\;  \mathcal{T}_{n}^\epsilon\leq t \right\}.
\end{equation}
Recall that  $N^\epsilon(t)$ is  the number of reflections made by the light ray before time $t$, while $\mathcal{T}_n^\epsilon$ represents the time of the $n$-th reflection.
We have
\begin{equation}\label{o13.3b}
Q(\mathcal{T}_n^\epsilon)=\wp_\eps(\bal_n^\eps,\bs^\eps_n)\exp\left(i\bal^\eps_n\right)
= \Gamma_\eps^{\bs^\eps_n}\left(\bal_n^\eps\right).
\end{equation}

\begin{lemma}\label{convex} There exists $\epsilon_0\in (0,1/4)$ such that for all $\epsilon \in (0,\epsilon_0)$, $\mathcal{U}_\epsilon^0$ and $\mathcal{U}_\epsilon^1$ are strictly convex.
\end{lemma}
\begin{proof}
For $j=0,1$, $\Gamma_\epsilon^j$ is a closed simple curve, so it suffices to show that there exists $\epsilon_0>0$ such that for all $\epsilon \in (0,\epsilon_0)$, its curvature is strictly positive at every point.
 
Standard calculations show that the curvature of $\Gamma_\epsilon^0$ at $\Gamma_\epsilon^0(\alpha)$ is given by 
\begin{align}\label{curvature}
\kappa_\epsilon^0(\alpha)&:=\frac{(1+ g_\epsilon(\alpha))\big(1+g _\epsilon(\alpha)-g_\epsilon''(\alpha)\big)+2 g_\epsilon'(\alpha)^2}{\big((1+g_\epsilon(\alpha))^2+ g_\epsilon'(\alpha)^2\big)^{3/2}},
\end{align}
while the curvature of $\Gamma_\epsilon^1$ at $\Gamma_\epsilon^1(\alpha)$ is given by
\begin{align}\label{curvature2}
\kappa_\epsilon^1(\alpha)&:=\frac{(1- f_\epsilon(\alpha))\big(1- f_\epsilon(\alpha)+ f_\epsilon''(\alpha)\big)+2 f_\epsilon'(\alpha)^2}{\big((1- f_\epsilon(\alpha))^2+ f_\epsilon'(\alpha)^2\big)^{3/2}}.
\end{align}

 By assumptions $\bold{H1}$-$\bold{H4}$, it is clear that $\kappa_\epsilon^s$ converges uniformly to $1$ as $\epsilon$ goes to $0$, for $s=0,1$. 
\end{proof}

 From now on, we will assume that $\epsilon\in (0,\epsilon_0)$, where $\eps_0$ is given in Lemma \ref{convex}, so that $\mathcal{U}_\epsilon^0$ and $\mathcal{U}_\epsilon^1$ are  strictly convex for all values of the parameter $\eps$. This implies that when the light ray reflects from $\Gamma_\epsilon^1$, the next reflection is from $\Gamma_\epsilon^0$. On the other hand, when the light ray reflects from $\Gamma_\epsilon^0$ at $\Gamma_\epsilon^0(\alpha)$, there is a strictly positive probability $p_\epsilon(\alpha)$ that the next reflection point is again on $\Gamma_\epsilon^0$. In other words,
\begin{align*}
p_\epsilon(\alpha) = \P( \bs^\eps_{n+1} = 0 \mid \bs^\eps_n =0, 
\bal_n^\eps = \alpha)>0.
\end{align*}
 
Since $\Gamma_\epsilon^0$ and $\Gamma_\epsilon^1$ are not necessary circles,  $p_\epsilon(\alpha)$ may depend on the reflection point $Q(\mathcal{T}_n^\epsilon)$. 

\begin{definition}\label{o18.1b}

We will define some random variables for $n\geq 1$, $\alpha \in [0,2\pi)$ and $\eps\in(0,\eps_0)$.

Let $T^\eps_n(\alpha)$ be a random variable with the distribution of 
$\bal^\eps_1 - \bal^\eps_0$ conditioned on $\{\bs^\eps_0 = 1, \bal^\eps_0=\alpha\}$, i.e., on the event that the light ray starts from $\Gamma^1_\eps(\alpha) $.

Let $R^\eps_n(\alpha)$ be a random variable with the distribution of 
$\bal^\eps_1 - \bal^\eps_0$ conditioned on $\{\bs^\eps_0 = 0, \bs^\eps_1= 1, \bal^\eps_0=\alpha\}$, i.e., on the event that the light ray starts from $\Gamma^0_\eps(\alpha) $ and the next reflection is on the inner boundary.

Let $S^\eps_n(\alpha)$ be a random variable with the distribution of 
$\bal^\eps_1 - \bal^\eps_0$ conditioned on $\{\bs^\eps_0 = 0, \bs^\eps_1= 0, \bal^\eps_0=\alpha\}$, i.e., on the event that the light ray starts from $\Gamma^0_\eps(\alpha) $ and the next reflection is also on the outer boundary.

Let $\Lambda^\eps_n(\alpha)$ be a random variable with the distribution given by
$\P(\Lambda^\eps_n(\alpha) = 1)=1- \P(\Lambda^\eps_n(\alpha) = 0)=p_\eps(\alpha)$.

We assume that all  random variables listed above, for all $n\geq 1$, $\alpha \in [0,2\pi)$ and $\eps\in(0,\eps_0)$,  are jointly independent.
\end{definition}

The process $\{(\bal^\eps_n,\bs^\eps_n), n\geq 0\}$ can be represented as follows. For $n\geq 0$,
\begin{align}\label{InducedMCgeneralise}
\bs_{n+1}^\epsilon &= (1-\bs_n^\epsilon)(1-\Lambda_{n+1}^\epsilon(\bal_n^\epsilon)),\\
\bal_{n+1}^\epsilon &= \bal_n^\epsilon + T_{n+1}^\epsilon(\bal_n^\epsilon) \bs_n^\epsilon(\bal_n^\epsilon) \notag\\
&\quad+(1-\bs_n^\epsilon(\bal_n^\epsilon))\left(\Lambda_{n+1}^\epsilon(\bal_n^\epsilon) S_{n+1}^\epsilon(\bal_n^\epsilon) +(1-\Lambda_{n+1}^\epsilon(\bal_n^\epsilon))R_{n+1}^\epsilon (\bal_n^\epsilon)\right).\notag
\end{align}

\begin{lemma}\label{maxDep} 
For all $n\geq 0$ and $0<\eps<1/2$, a.s., $\vert \bal_{n+1}^\epsilon-\bal_n^\epsilon\vert \leq 12\sqrt{\epsilon}$. 
\end{lemma}
\begin{proof}
Let $\calB((x,y),r)$ denote the open disc with center $(x,y)$ and radius $r$.
The estimate  follows easily from the argument in the proof of Lemma \ref{Support} and the fact that $\calB\left((0,0),1-2\epsilon\right)\subset \mathcal{D}_\epsilon\subset \overline\calB\left((0,0),1+\epsilon\right)$.
\end{proof}

For $\alpha\in\mathbb{R}$ and $s\in\{0,1\}$, let $\gamma_\epsilon^s(\alpha)$ denote the angle between the inner normal vector in $\mathcal{D}_\epsilon$ at $\Gamma_\epsilon^j(\alpha)$ and the vector $(2s-1)\wp_\eps(\alpha,s) e^{i\alpha}$. The latter vector goes from $\wp_\eps(\alpha,s) e^{i\alpha}$ to $(0,0)$, so it has the same direction as $- e^{i\alpha}$.
By convention, we choose the sign of $\gamma_\epsilon^s(\alpha)$ so that it is positive if $s=0$ and $g'_\eps(\alpha) >0$, or  $s=1$ and $f'_\eps(\alpha) >0$. This means that if $\gamma_\epsilon^s(\alpha)>0$ for both $s=0$ and $s=1$ then  $\mathcal{D}_\epsilon$ is locally widening in the direction of increasing $\alpha$.

\begin{lemma}\label{angleecart} 
We have
\begin{align}\label{d4.1}
\gamma_\epsilon^0(\alpha)&= \arcsin\left(\frac{g_\epsilon'(\alpha)}{\sqrt{(1+g_\epsilon(\alpha))^2+g_\epsilon'(\alpha)^2}}\right),\\
\gamma_\epsilon^1(\alpha)&= \arcsin\left(\frac{f_\epsilon'(\alpha)}{\sqrt{(1-f_\epsilon(\alpha))^2+f_\epsilon'(\alpha)^2}}\right).
\label{d4.2}
\end{align}
\end{lemma}

\begin{proof} 
Let $\langle \,.\, , \,.\,\rangle$ denote the scalar product. Then 
\begin{align*}
\gamma_\epsilon^0(\alpha)
= \arcsin \left\langle \frac{ (\Gamma_\epsilon^0) '(\alpha) }{ |(\Gamma_\epsilon^0) '(\alpha)|} , e^{i\alpha}\right\rangle.
\end{align*}
We have
\begin{align*}
(\Gamma_\epsilon^0) '(\alpha)= 
\frac{d}{d\alpha} \left((1+g_\epsilon(\alpha)) e^{i\alpha}\right)
= (1+g_\epsilon(\alpha)) i e^{i\alpha} + g'_\epsilon(\alpha) e^{i\alpha},
\end{align*}
so that 
\begin{equation}
\left\langle (\Gamma_\epsilon^0) '(\alpha)  , e^{i\alpha}\right\rangle= g_\epsilon'(\alpha).
\end{equation}
Thus 
\begin{align*}
\gamma_\epsilon^0(\alpha)
= \arcsin \left\langle \frac{ (\Gamma_\epsilon^0) '(\alpha) }{ |(\Gamma_\epsilon^0) '(\alpha)|} , e^{i\alpha}\right\rangle
=  \arcsin\left(\frac{g_\epsilon'(\alpha)}{\sqrt{(1+g_\epsilon(\alpha))^2+g_\epsilon'(\alpha)^2}}\right).
\end{align*}
This proves \eqref{d4.1}. The proof of \eqref{d4.2} is analogous.
\end{proof}
\begin{lemma}\label{probReste} For some $\eps_1>0$, all $\epsilon\in (0,\epsilon_1)$ and all $\alpha\in \R$,
\begin{equation}\label{d2.2}
\eps/2 \leq p_\epsilon(\alpha) \leq  \eps\left(4 +6 \|g'\|_\infty\right).
\end{equation}
\end{lemma} 
\begin{proof}

Let $L$ be the straight line passing through $\Gamma_\epsilon^0(\alpha)$ and orthogonal to $\Gamma_\epsilon^0$ at this point. It
follows from \bhth{} and Lemma \ref{angleecart} that for some $\eps_2>0$ and all $\eps\in(0,\eps_2)$, the 
angle $\theta_\eps$ between $L$ and the line segment with endpoints $\Gamma_\epsilon^0(\alpha)$ and $(0,0)$ is less  than $2\eps \|g'\|_\infty$.
Hence, for some $\eps_3>0$ and all $\eps\in(0,\eps_3)$, the distance between $L$ and $(0,0)$ is less than $3\eps \|g'\|_\infty$. Let $x$ be the point in $L$ closest to $(0,0)$. 
Thus, $\dist(x, (0,0)) <3\eps \|g'\|_\infty$ for small $\eps$.
Assumption \bho{} and \eqref{d2.1} imply that the circle $\calC(x, 1- \eps(2 +3 \|g'\|_\infty) ) $ lies inside $\mathcal{U}_\epsilon^1$. 
Hence, a light ray starting from $\Gamma_\epsilon^0(\alpha)$ will hit $\Gamma_\epsilon^1$ before hitting the circle $\calC(x, 1- \eps(2 +3 \|g'\|_\infty) ) $. The Lambertian direction of the light ray starting from $\Gamma_\epsilon^0(\alpha)$, defined as in \eqref{LambertianDistrib}, is the same whether it is defined relative to  $\mathcal{D}_\epsilon$ or the interior of $\calC(x, |\Gamma_\epsilon^0(\alpha) - x| ) $ because the boundaries of the two domains are tangent at $\Gamma_\epsilon^0(\alpha)$. We can apply 
Lemma \ref{remain} to the domain between the circles $\calC(x, 1- \eps(2 +3 \|g'\|_\infty) ) $ and $\calC(x, |\Gamma_\epsilon^0(\alpha) - x| ) $. 
By \bho{}, for small $\eps>0$,
\begin{align}\label{j13.1}
1- 3\eps \|g'\|_\infty \leq
|\Gamma_\epsilon^0(\alpha) - x|  \leq |\Gamma_\epsilon^0(\alpha) - (0,0)| 
+ |(0,0) - x| \leq  1+ \eps +3\eps \|g'\|_\infty.
\end{align}
Lemma \ref{remain}, \eqref{j13.1}
and rescaling by the factor of $|\Gamma_\epsilon^0(\alpha) - x|$ imply that,
for sufficiently small $\eps>0$,
\begin{align*}
p_\eps(\alpha) &\leq
\frac{|\Gamma_\epsilon^0(\alpha) - x| -( 1- \eps(2 +3 \|g'\|_\infty) ) }
{|\Gamma_\epsilon^0(\alpha) - x|}
\leq \frac{ 1+ \eps +3\eps \|g'\|_\infty -( 1- \eps(2 +3 \|g'\|_\infty) ) }
{1- 3\eps \|g'\|_\infty}\\
&= \frac{  \eps(3 +6 \|g'\|_\infty) ) }
{1- 3\eps \|g'\|_\infty} \leq \eps(4 +6 \|g'\|_\infty) ).
\end{align*}

Assumption \bho{} and \eqref{d2.1} imply that the circle $\mathcal{U}_\epsilon^1$ lies inside the circle  $\calC((0,0), 1- \eps ) $,
and $ |\Gamma_\epsilon^0(\alpha) - (0,0)| \geq  1 $.
Since the line $L$ does not have to pass through $(0,0)$, we can use only one half of the estimate in
Lemma \ref{remain} to conclude that for a light ray starting from $\Gamma_\epsilon^0(\alpha)$ with the Lambertian direction, the probability of avoiding of $\calC((0,0), 1- \eps ) $ is bounded below by $\eps/2$. Hence, the probability of avoiding $\mathcal{U}_\epsilon^1$ is also bounded below by $\eps/2$. This proves the lower bound.
\end{proof}
\begin{lemma}\label{EstimT} The following assertions hold uniformly in $\alpha\in[0,2\pi)$,
\begin{align}\label{MeanT}
\lim_{\eps\rightarrow 0}\frac{\mathbb{E}\big(T_1^\eps(\alpha)\big)}{(\epsilon^{2}/2)\log(1/\eps)}&= h'(\alpha)h(\alpha),\\
\label{VarT}
\lim_{\eps\rightarrow 0}\frac{\var(T_1^\eps(\alpha))}{(\epsilon^2/2)\log(1/\eps)}&=h^2(\alpha).
\end{align}
\end{lemma}

\begin{proof} 
(i)
We will prove the lemma for $\alpha = \pi/2$. This will cause no loss of generality because the constants  in our estimates do not dependent on $\alpha$.

Let $L(\alpha)$ be the straight line passing through $\Gamma^0_\eps(\alpha)$ in the direction of the normal vector to $\Gamma^0_\eps$ at $\Gamma^0_\eps(\alpha)$, for $\alpha \in [0,2\pi)$.

It follows from \bhth{} that for some $c_1<\infty$ and all $\eps\in(0,1/2)$, we have $\|g'_\eps\|_\infty < c_1 \eps$.
Let $c_ 2 = 30  c_1$ and 
$\alpha_1 = \pi/2 - c_2 \eps^2$.
The unsigned angle  between $L(\alpha_1)$ and the vertical axis
is $\rho_0:=\pi/2-\alpha_1 + \arctan\big(\frac{g'_\eps(\alpha_1)}{1+g_\eps(\alpha_1)}\big)$. For small $\eps>0$, $\rho_0\leq 2 c_1 \eps $, so $\tan \rho_0 \leq 3 c_1 \eps$. 
It is possible that $L(\alpha_1)$ does not cross the vertical axis below $\Gamma^0_\eps(\alpha_1)$. Suppose that it does and denote by $(0,u_1)$ the intersection point with the vertical axis. Then
\begin{align*}
\big(1+g_\eps(\alpha_1)\big)\cos(c_2\eps^2)-u_1&=\frac{\big(1+g_\eps(\alpha_1)\big)\sin(c_2\eps^2)}{\tan(\rho_0)}.
\end{align*}
Therefore, by \bho{}, for small $\eps>0$,
\begin{align*}
1-u_1&=\frac{\big(1+g_\eps(\alpha_1)\big)\sin(c_2\eps^2)}{\tan(\rho_0)}-g_\eps(\alpha_1)\cos(c_2\eps^2)+1-\cos(c_2\eps^2)\\
&\geq \frac{c_2\eps^2/2}{\tan(\rho_0)}-\Vert g_\eps\Vert_\infty
\geq \left(\frac{c_2}{6c_1}-2\right)\eps =3\eps.
\end{align*}
Since $f_\eps(\pi/2) \leq 2 \eps$,  $L(\alpha_1)$ crosses the vertical axis 
below $\Gamma^1_\eps(\pi/2)$, and stays to the right of $\Gamma^1_\eps(\pi/2)$ above $(0, u_1)$.
We have shown that no matter whether $L(\alpha_1)$  crosses the vertical axis below $\Gamma^0_\eps(\alpha_1)$ or not, it stays to the right of $\Gamma^1_\eps(\pi/2)$.

An analogous argument shows that if $\alpha_2 = \pi/2 + c_2 \eps^2$  then $L(\alpha_2)$ stays to the left of $\Gamma^1_\eps(\pi/2)$.
Since $g_\eps$ is $C^3$, 
when $\alpha$ varies continuously from $\pi/2  - c_2\eps^2$ to $ \pi/2 + c_2\eps^2$, we must encounter $\alpha'$ such that $L(\alpha')$ passes through  $\Gamma^1_\eps(\pi/2)$.
We have 
\begin{align}\label{d15.0}
|\pi/2 - \alpha'| = O( \eps^2).
\end{align}
Let $\calR$ be the rotation about $\Gamma^1_\eps(\pi/2)$, with the angle of rotation $\rho$ chosen so that $\calR(\Gamma^0_\eps(\alpha'))=(0, z)$ with $z\geq 1$. 
We will estimate $\rho$. The angle between $L(\alpha')$ and the line segment between $0$ and $e^{i\alpha'}$ is equal to $\gamma^0_\eps(\alpha')$, by definition. Hence, the angle between $L(\alpha')$ and the vertical line (i.e., $\rho$) is
$\gamma^0_\eps(\alpha')+ (\pi/2-\alpha')$. Recall that $\|g'_\eps\|_\infty < c_1 \eps$ and a similar bound holds for $g_\eps$. This, \eqref{d4.1}, \bho-\bhf, and the Taylor expansion imply that
\begin{align}\label{d15.1}
\rho &= \gamma^0_\eps(\alpha')+ (\pi/2-\alpha' )
= g'_\eps(\alpha')(1 + O(\eps)) + O(\eps^2)
= g'_\eps(\alpha') + O(\eps^2)\\
&= g'_\eps(\pi/2)(1+ O(\eps(\pi/2-\alpha' ))) + O (\eps^2)
= g'_\eps(\pi/2) + O (\eps^2) = O(\eps). \nonumber
\end{align}

Let $\calC$ be the osculating circle of $\calR(\Gamma^0_\eps)$ at $(0,z)$.
We have chosen $\rho$ so that the topmost point of the circle is on the vertical axis. The radius of $\calC$ is $1/ |\kappa^0_\eps(\alpha')|$ (see \eqref{curvature} for a formula for the curvature $\kappa^0_\eps$), so the center of $\calC$ is at $(0,z_1) :=(0, z - 1/ |\kappa^0_\eps(\alpha')|)$.
We will now estimate $z$ and $z_1$. The definition of $z$ and the formula for the image of $\Gamma^0_\eps(\alpha') = ((1+g_\eps(\alpha'))\cos(\alpha'), (1+g_\eps(\alpha'))\sin(\alpha'))$  under rotation about $(0,1-f_\eps(\pi/2))$  by angle $\rho$ yield
\begin{align*}
z&= (1-f_\eps(\pi/2))(1-\cos(\rho))+\sin(\rho)(1+g_\eps(\alpha'))\cos(\alpha')+\cos(\rho)(1+g_\eps(\alpha'))\sin(\alpha')\\
&=(1-f_\eps(\pi/2))(1-\cos(\rho))
+ (1+g_\eps(\alpha'))\sin( \rho+\alpha') .
\end{align*} 
This, \bho{}, \bhth{}, \eqref{d15.0} and \eqref{d15.1} show that 
\begin{align}
|z-(&1+g_\eps(\pi/2))|
= \left|(1-f_\eps(\pi/2))(1-\cos(\rho))
+ (1+g_\eps(\alpha'))\sin(\rho + \alpha')
-(1+g_\eps(\pi/2))\right| \notag \\
&=|(1-f_\eps(\pi/2))(1-\cos(\rho))
+ (1+g_\eps(\alpha'))
+ (1+g_\eps(\alpha'))(\sin(\rho + \alpha')-1) \notag \\
&\quad-(1+g_\eps(\pi/2))| \notag \\
&=\left|(1-f_\eps(\pi/2))(1-\cos(\rho))
+ (g_\eps(\alpha')-g_\eps(\pi/2))
+ (1+g_\eps(\alpha'))(\sin(\rho + \alpha')-1)\right| \notag \\
&=\left|(1-f_\eps(\pi/2))O(\rho^2)
+ (g_\eps(\alpha')-g_\eps(\pi/2))
+ (1+g_\eps(\alpha'))O((\pi/2- \alpha'-\rho)^2)\right| \notag \\
&= O(\eps^2) + O(\eps^2) + O(\eps^2) = O(\eps^2).\label{d17.1}
\end{align}
 Assumptions \bho-\bhf, \eqref{curvature} and the Taylor expansion imply that $|\kappa^0_\eps(\alpha)| = 1 + O(\eps)$ uniformly in $\alpha$. We combine this with \bht{} and \eqref{d17.1} to see that
\begin{align}\label{d17.2}
|z_1| =|z - 1/ |\kappa^0_\eps(\alpha')|| = O(\eps).
\end{align}

Suppose that a light ray leaves $\Gamma^1_\eps(\pi/2)$ at an angle $\theta$, relative to the normal vector to $\Gamma^1_\eps$ at $\Gamma^1_\eps(\pi/2)$. The light ray will intersect $\calR(\Gamma^0_\eps)$ at a point that we will denote $r'(\theta) \exp(i (\pi/2+ T'(\theta)))$. In other words, $T'(\theta)$ denotes the angular distance between $\Gamma^1_\eps(\pi/2)$ and the intersection of the light ray with $\calR(\Gamma^0_\eps)$. 
The same light ray will intersect the circle $\calC$ at a point $\wh r(\theta) \exp(i(\pi/2+ \wh T(\theta)))$.

For later reference, we record the following estimates valid for all $\theta$. They follow from an argument similar to the one used in the proof of Lemma \ref{maxDep}.
\begin{align}\label{d18.1}
|T^\eps_1| = O(\eps^{1/2}),
\qquad |T'(\theta)| = O(\eps^{1/2}),
\qquad |\wh T(\theta)| = O(\eps^{1/2}).
\end{align}

If we recall the notation from \eqref{LambertianDistrib} and write $T^\eps_1(\pi/2, \theta) = T^\eps_1(\pi/2)$ to emphasize the dependence on $\theta$, then
\begin{align}\label{d7.1}
T'(\theta+\rho) = T^\eps_1(\pi/2, \theta ).
\end{align}

The curvature of $\calC$ matches that of $\calR(\Gamma^0_\eps)$ at $(0,z)$, by the definition of the osculating circle. Hence, if $v\in \calC$ and $\dist(v, (0,z)) = b< 10 \sqrt{\eps} $  then for some $c_1<\infty$ (not depending on our choice of $\bal^\eps_0=\pi/2$, in view of \bhfi), the distance from $v$ to $\calR(\Gamma^0_\eps)$ is less than $c_1 b^3$.
This and an elementary analysis of the triangle with vertices $\wh T(\theta) $, $T'(\theta)$ and $\Gamma^0_\eps(\wh T(\theta))$ shows that 
\begin{align}\label{d18.5}
\left| \wh T(\theta) - T'(\theta)\right| = O\left( \wh T(\theta)^3 \tan \theta \right).
\end{align}
We will need a stronger version of this estimate for $\theta \leq -\pi/2 +c_3\eps^{1/2}$ and $\theta \geq \pi/2 -c_3\eps^{1/2}$.
If $\theta$ is in this range, $|\wh T(\theta)|\geq c_4 \eps^{1/2}$. It follows that, for some $c_5>0$, the slope of the osculating circle $\calC$ at $\wh r(\theta) \exp(i(\pi/2+ \wh T(\theta)))$, considered to be the graph of a function in the usual coordinate system, is greater than $c_5 \eps^{1/2}$ for $\theta \leq -\pi/2 +c_3\eps^{1/2}$ and smaller than $-c_5 \eps^{1/2}$ for $\theta \geq \pi/2 -c_3\eps^{1/2}$. The same remark applies to the slope of  $\calR(\Gamma^0_\eps)$ at $r'(\theta) \exp(i (\pi/2+ T'(\theta)))$. Hence, for $\theta \leq -\pi/2 +c_3\eps^{1/2}$ and $\theta \geq \pi/2 -c_3\eps^{1/2}$,
\begin{align}\label{d20.1}
\left| \wh T(\theta) - T'(\theta)\right| = O\left( \wh T(\theta)^3 \eps^{-1/2} \right).
\end{align}

We will write 
\begin{align}\label{d26.1}
(x,y) = (x(\theta), y(\theta)) =\wh r(\theta) \exp(i \wh T(\theta))
\end{align}
and we will find a formula for $x$ in terms of $\theta$, $f_\eps$ and $g_\eps$. If we let $a = 1/\tan \theta$ then $y(\theta) = a x(\theta) + 1 - f_\eps(\pi/2)$. Since $(x,y) \in \calC((0,z_1) ,1/ |\kappa^0_\eps(\alpha')|)$,
\begin{align*}
x^2 + y^2 
= x^2 + ( a x + 1 - f_\eps(\pi/2)-z_1)^2
= \kappa^0_\eps(\alpha')^{-2}.
\end{align*}
This and $a = 1/\tan \theta$ yield for $a>0$,
\begin{align}\label{d7.2}
x&(\theta)= 
\frac{-a(1-f_\eps(\pi/2)-z_1)+\sqrt{(a^2+1)\kappa^0_\eps(\alpha')^{-2}-(1-f_\eps(\pi/2)-z_1)^2}}{1+a^2} \\
&= \sin \theta \cos \theta
(f_\eps(\pi/2)+z_1 - 1) \nonumber \\
&\qquad+ 
\sin \theta \cos \theta
\sqrt{\kappa^0_\eps(\alpha')^{-2}+ \tan^2\theta
\left(\kappa^0_\eps(\alpha')^{-2}-(1-f_\eps(\pi/2)-z_1)^2\right)}
.\nonumber
\end{align}
We have 
\begin{align}\label{d15.4}
\wh T(\theta) = \arctan\left( x(\theta)/y(\theta)\right)
= \arctan\left( \frac{x(\theta)}{ x(\theta)/\tan \theta + 1 - f_\eps(\pi/2)} \right).
\end{align}

The density of  the angle of reflection given  in \eqref{LambertianDistrib} is relative to the normal vector at the boundary of the domain, 
which is tilted by $\gamma_\eps^1(\pi/2)$ relative to the vertical
if the reflection takes place at $\Gamma^1_\eps(\pi/2)$,
so
\begin{align}\label{d18.6}
\E  \wh T(\Theta+\rho) = \int_{-\pi/2}^{\pi/2} 
\wh T\left(\theta+\rho+ \gamma_\eps^1(\pi/2)\right) \frac 12 \cos \theta d\theta.
\end{align}
Let $\rho_1 = \rho+ \gamma_\eps^1(\pi/2)$. We will assume that $\rho_1 \geq 0$. The argument is analogous in the opposite case.
We have
\begin{align}\label{d18.2}
\E  \wh T(\Theta+\rho) 
&= \int_{-\pi/2}^{\pi/2} 
\wh T\left(\theta+\rho_1\right) \frac 12 \cos \theta d\theta\\
&= \int_{-\pi/2}^{\pi/2-2\rho_1} 
\wh T\left(\theta+\rho_1\right) \frac 12 \cos \theta d\theta
+ \int_{\pi/2-2\rho_1}^{\pi/2} 
\wh T\left(\theta+\rho_1\right) \frac 12 \cos \theta d\theta.\nonumber
\end{align}
We will estimate the two integrals separately.
We start with the first integral.
\begin{align}\label{d19.4}
\int_{-\pi/2}^{\pi/2-2\rho_1} 
\wh T\left(\theta+\rho_1\right) \frac 12 \cos \theta d\theta
= \int_{-\pi/2+\rho_1}^{\pi/2-\rho_1} 
\wh T\left(\theta\right) \frac 12 \cos (\theta-\rho_1) d\theta.
\end{align}
Recall that $\cos (\theta-\rho_1) = \cos \theta \cos \rho_1 + \sin \theta\sin\rho_1$. Note that $\theta \to \wh T\left(\theta\right)$ is an odd function. Thus
\begin{align}\label{d15.2}
\int_{-\pi/2}^{\pi/2-2\rho_1} 
\wh T\left(\theta+\rho_1\right) \frac 12 \cos \theta d\theta
&= \frac 12\sin\rho_1 \int_{-\pi/2+\rho_1}^{\pi/2-\rho_1} 
\wh T\left(\theta\right)  \sin \theta d\theta\\
&= \sin\rho_1 \int_{0}^{\pi/2-\rho_1} 
\wh T\left(\theta\right)  \sin \theta d\theta.\nonumber
\end{align}
Once again, we will analyze the factors on the right hand side separately.
First, by \eqref{d15.1} and Lemma \ref{angleecart},
\begin{align}\label{d17.10}
\lim_{\eps\to 0} \frac 1 \eps \sin\rho_1
&= \lim_{\eps\to 0} \frac 1 \eps \sin(\rho+ \gamma_\eps^1(\pi/2))
= \lim_{\eps\to 0} \frac 1 \eps \sin
\left(\gamma^0_\eps(\pi/2) + O (\eps^2)+ \gamma_\eps^1(\pi/2)\right)\\
&= \lim_{\eps\to 0} \frac 1 \eps \sin
\left(\gamma^0_\eps(\pi/2) + \gamma_\eps^1(\pi/2)\right) = h'(\pi/2).
\nonumber
\end{align}
Next we tackle the integral on the right hand side of \eqref{d15.2}.
It is easy to see the $y(\theta)$ converges to 1 and $x(\theta)$ converges to 0, both uniformly in $\theta$, as $\eps\to0$. This observation and \eqref{d15.4} imply that 
\begin{align}\label{d18.10}
\lim_{\eps\to 0} \wh T(\theta)/x(\theta) = 1,
\end{align}
 uniformly in $\theta$, so
\begin{align}\label{d15.3}
\lim_{\eps\to0} 
\frac 1 {\eps \log(1/\eps)}
\int_{0}^{\pi/2-\rho_1} 
\wh T\left(\theta\right)  \sin \theta d\theta
= \lim_{\eps\to0} 
\frac 1 {\eps \log(1/\eps)}
\int_{0}^{\pi/2-\rho_1} 
x(\theta)  \sin \theta d\theta,
\end{align}
assuming that at least one of these limits exists.

In order to estimate the integral on the right hand side of \eqref{d15.3}, we will split the interval of integration into two parts.
Set $h_\eps =f_\eps +g_\eps$ and
\begin{equation*}
\theta_0=\frac{\pi}{2}-\sqrt{\eps}.
\end{equation*}

An easy argument, similar to the one  in the the proof of Lemma \ref{Support}, shows that for some $c_3$, all $\eps\in(0,1/2)$ and all $\theta$, $|x(\theta)|  \leq c_3 \sqrt{\eps}$. Hence
\begin{align*}
\left|\int_{\theta_0}^{\pi/2-\rho_1} 
x(\theta)  \sin \theta d\theta\right|
\leq \int_{\theta_0}^{\pi/2} 
|x(\theta)|   d\theta
\leq  c_3 \sqrt{\eps}  (\pi/2 - \theta_0)
= O(\eps),
\end{align*}
and, therefore,
\begin{align}\label{d16.2}
\lim_{\eps\to0} 
\frac 1 {\eps \log(1/\eps)}
\left|\int_{\theta_0}^{\pi/2-\rho_1} 
x(\theta)  \sin \theta d\theta\right|
=0.
\end{align}

Recall from \eqref{d7.2} that
\begin{align}\label{d16.1}
x(\theta)
&= \sin \theta \cos \theta
(f_\eps(\pi/2)+z_1 - 1)  \\
&\qquad+ 
\sin \theta \cos \theta
\sqrt{\kappa^0_\eps(\alpha')^{-2}+ \tan^2\theta
\left(\kappa^0_\eps(\alpha')^{-2}-(1-f_\eps(\pi/2)-z_1)^2\right)}
.\nonumber
\end{align}
It follows from \eqref{d17.1} that
\begin{align}\label{d23.1}
|\kappa^0_\eps(\alpha')|^{-1} = z - z_1
= 1+ g_\eps(\pi/2) - z_1 + O(\eps^2),
\end{align}
so
\begin{align}\label{d22.1}
\kappa^0_\eps(\alpha')^{-2} 
= (1+ g_\eps(\pi/2) - z_1)^2 + O(\eps^2).
\end{align}
This implies the following representation for the expression under square root in \eqref{d16.1},
\begin{align}\label{d17.3}
&\sqrt{\kappa^0_\eps(\alpha')^{-2}+ \tan^2\theta
\left(\kappa^0_\eps(\alpha')^{-2}-(1-f_\eps(\pi/2)-z_1)^2\right)}\\
&=
\Big((1+ g_\eps(\pi/2) - z_1)^2 + O(\eps^2)\nonumber \\
&\qquad+ \tan^2\theta
\left((1+ g_\eps(\pi/2) - z_1)^2 + O(\eps^2)-(1-f_\eps(\pi/2)-z_1)^2\right)\Big)^{1/2}.\nonumber
\end{align}
It follows from \bht{} and \eqref{d17.2} that
\begin{align}\label{d17.4}
(&1+ g_\eps(\pi/2) - z_1)^2 + O(\eps^2)-(1-f_\eps(\pi/2)-z_1)^2 \\
&= g_\eps(\pi/2) ^2 -f_\eps(\pi/2)^2 
+ 2 (1- z_1) (g_\eps(\pi/2)  +f_\eps(\pi/2)) + O(\eps^2) \nonumber\\
&= 2 (1- z_1) h_\eps(\pi/2)  + O(\eps^2)
= O(\eps).\label{d17.5}
\end{align}
Since $\tan\theta = O(\eps^{1/2})$ for $0 \leq \theta \leq\theta_0$,  the above estimate implies that
\begin{align*}
\tan^2\theta
\left((1+ g_\eps(\pi/2) - z_1)^2 + O(\eps^2)-(1-f_\eps(\pi/2)-z_1)^2\right)
=O(\eps^2).
\end{align*}
This and \eqref{d17.4}-\eqref{d17.5} imply that we can apply the Taylor expansion to the right hand side of \eqref{d17.3} as follows,
\begin{align*}
&\Big((1+ g_\eps(\pi/2) - z_1)^2 + O(\eps^2)\\
&\qquad+ \tan^2\theta
\left((1+ g_\eps(\pi/2) - z_1)^2 + O(\eps^2)-(1-f_\eps(\pi/2)-z_1)^2\right)\Big)^{1/2}\\
& = 1+ g_\eps(\pi/2) - z_1
+ \frac {\tan^2\theta \left( 2 (1+ g_\eps(\pi/2)- z_1) h_\eps(\pi/2)  + O(\eps^2)\right)} {2(1+ g_\eps(\pi/2) - z_1)}
+ O(\tan^4\theta \ h_\eps(\pi/2)^2)\\
& = 1+ g_\eps(\pi/2) - z_1
+ \tan^2\theta\  h_\eps(\pi/2)
+ \tan^2\theta\  O(\eps^2)
+ O(\tan^4\theta \ h_\eps(\pi/2)^2)\\
& = 1+ g_\eps(\pi/2) - z_1
+ \tan^2\theta\  h_\eps(\pi/2)
+ (\tan^2\theta+\tan^4\theta)  O(\eps^2).
\end{align*}
We combine this with \eqref{d16.1}, \eqref{d17.3} and \eqref{d17.5} to obtain
\begin{align}\label{d18.9}
x(\theta) 
 = \sin\theta \cos\theta\left(h_\eps(\pi/2) (1 + \tan^2\theta)
+ (\tan^2\theta+\tan^4\theta)  O(\eps^2)\right).
\end{align}
Hence,
\begin{align}\nonumber
\int_0^{\theta_0} x(\theta) \sin\theta d\theta
&=
\int_0^{\theta_0} \sin\theta \cos\theta\left(h_\eps(\pi/2) (1 + \tan^2\theta)
+ (\tan^2\theta+\tan^4\theta)  O(\eps^2)\right) \sin\theta d\theta\\
&= h_\eps(\pi/2) \int_0^{\theta_0} \frac{\sin^2\theta}{\cos\theta} d\theta
+ O(\eps^2) \int_0^{\theta_0} \left( \frac{\sin^4\theta}{\cos\theta} 
+ \frac{\sin^6\theta}{\cos^3\theta} \right)  d\theta.\label{d17.8}
\end{align}
We use \eqref{d17.6}-\eqref{d17.7} in the following calculation,
\begin{align*}
 h_\eps&(\pi/2) \int_0^{\theta_0} \frac{\sin^2\theta}{\cos\theta} d\theta
= - \frac{h_\eps(\pi/2)}{2} \log\left(1-\cos (\pi/2-\theta_0)\right)\\
&= - \frac{h_\eps(\pi/2)}{2} \log\left(1-\cos \left(\sqrt{\eps}\right)\right)
=  \frac{h_\eps(\pi/2)}{2} \log\left(\frac{1+o(1)}{\eps}\right).
\end{align*}
Thus, in view of \bht,
\begin{align}\label{d17.9}
\lim_{\eps\to0}\
&\frac 1 {\eps \log (1/\eps)}
 h_\eps(\pi/2) \int_0^{\theta_0} \frac{\sin^2\theta}{\cos\theta} d\theta
= 
\lim_{\eps\to0}\
\frac 1 {\eps \log (1/\eps)}
 \frac{h_\eps(\pi/2)}{2} \log\left(\frac{1+o(1)}{\eps}\right)\\
&=\lim_{\eps\to0}
\frac 1 2
\frac {h_\eps(\pi/2)} {\eps }
= h(\pi/2)/2.\nonumber
\end{align}
We use \eqref{o17.5} and \bht{} as follows,
\begin{align*}
O(\eps^2) \int_0^{\theta_0} \left( \frac{\sin^4\theta}{\cos\theta} 
+ \frac{\sin^6\theta}{\cos^3\theta} \right)  d\theta
\leq O(\eps^2) 2
\int_0^{\theta_0} \frac{1}{\cos^3\theta}  d\theta
\leq O(\eps^2) 2  \frac { \pi (\pi - \theta_0) \theta_0}{(\pi - 2\theta_0)^2}
= O(\eps),
\end{align*}
from which we conclude that
\begin{align*}
\lim_{\eps\to0}
&\frac 1 {\eps \log (1/\eps)}
O(\eps^2) \int_0^{\theta_0} \left( \frac{\sin^4\theta}{\cos\theta} 
+ \frac{\sin^6\theta}{\cos^3\theta} \right)  d\theta
=0.
\end{align*}
We combine this with \eqref{d17.8} and \eqref{d17.9}
to conclude that
\begin{align*}
\lim_{\eps\to0}
&\frac 1 {\eps \log (1/\eps)}
\int_0^{\theta_0} x(\theta) \sin\theta d\theta
=h(\pi/2)/2.
\end{align*}
Thus, in view of \eqref{d15.3} and \eqref{d16.2},
\begin{align*}
\lim_{\eps\to0} 
\frac 1 {\eps \log(1/\eps)}
\int_{0}^{\pi/2-\rho_1} 
\wh T\left(\theta\right)  \sin \theta d\theta
=h(\pi/2)/2.
\end{align*}
Combining the formula with \eqref{d15.2} and \eqref{d17.10} yields
\begin{align}\label{d18.3}
\lim_{\eps\to0} 
\frac 1 {\eps^2 \log(1/\eps)}
\int_{-\pi/2}^{\pi/2-2\rho_1} 
\wh T\left(\theta+\rho_1\right) \frac 12 \cos \theta d\theta
&= h'(\pi/2)h(\pi/2)/2.
\end{align}

We next estimate the second integral on the right hand side of \eqref{d18.2}. Recall that  $\rho_1 = \rho+ \gamma_\eps^1(\pi/2)$. We have 
\begin{align}\label{d19.3}
|\rho_1| = O(\eps)
\end{align}
in view of \eqref{d4.2} and \eqref{d15.1}.
Hence \eqref{d18.1} gives
\begin{align*}
 \int_{\pi/2-2\rho_1}^{\pi/2} 
\wh T\left(\theta+\rho_1\right) \frac 12 \cos \theta d\theta
=
O(\eps^{1/2}) \int_{\pi/2-2\rho_1}^{\pi/2} 
 \cos \theta d\theta
= O(\eps^{1/2}) O(\rho_1^2) = O(\eps^{5/2}).
\end{align*}
This, \eqref{d18.2} and \eqref{d18.3} yield
\begin{align}\label{d18.4}
\lim_{\eps\to0} 
\frac {\E \wh T(\Theta+\rho)} {\eps^2 \log(1/\eps)}
=
\lim_{\eps\to0} 
\frac 1 {\eps^2 \log(1/\eps)}
\int_{-\pi/2}^{\pi/2} 
\wh T\left(\theta+\rho_1\right) \frac 12 \cos \theta d\theta
&= h'(\pi/2)h(\pi/2)/2.
\end{align}

We will now estimate $\E\left| \wh T(\theta+\rho) - T^\eps_1(\pi/2, \theta )\right|$. By \eqref{d7.1}, \eqref{d18.5} and \eqref{d20.1}, for some $c_4$,
\begin{align}\label{d20.3}
\E& \left| \wh T(\Theta+\rho) - T^\eps_1(\pi/2, \Theta )\right|
= \E \left| \wh T(\Theta+\rho) - T'(\Theta+\rho)\right|\\
&\leq c_4
\E \left(\left| \wh T(\Theta+\rho)^3 \tan (\Theta+\rho) \right|
\bone_{(-\pi/2+\eps^{1/2} , \pi/2 -\eps^{1/2} -2\rho_1)}(\Theta) \right)
\notag\\
&\quad + c_4
\E \left(\left| \wh T(\Theta+\rho) ^3 \eps^{-1/2}\right|
\bone_{(-\pi/2 ,-\pi/2+\eps^{1/2} )\cup( \pi/2 -\eps^{1/2} -2\rho_1,\pi/2)}(\Theta) \right).\notag
\end{align}
By \eqref{d18.1},
\begin{align}\label{d20.2}
\E &\left(\left| \wh T(\Theta+\rho) ^3 \eps^{-1/2}\right|
\bone_{(-\pi/2 ,-\pi/2+\eps^{1/2} )\cup( \pi/2 -\eps^{1/2} -2\rho_1,\pi/2)}(\Theta) \right)\\
&\leq O(\eps) 
\left( \int_{-\pi/2 }^{-\pi/2+\eps^{1/2} }
+\int_{ \pi/2 -\eps^{1/2} -2\rho_1}^{\pi/2}
\right) \frac 12 \cos \theta d\theta
= O(\eps^{2}).\notag
\end{align}

We calculate as in \eqref{d18.2} and \eqref{d19.4},  use the fact that $\theta \to \left| \wh T(\theta)^3 \tan (\theta) \right|$ is even,
and then apply \eqref{d18.10},
\begin{align}\label{d19.5}
\E& \left(\left| \wh T(\Theta+\rho)^3 \tan (\Theta+\rho) \right|
\bone_{(-\pi/2+\eps^{1/2} , \pi/2 -\eps^{1/2} -2\rho_1)}(\Theta) \right)\\
&= \int_{-\pi/2+\eps^{1/2}}^{\pi/2-\eps^{1/2}-2\rho_1} 
\left| \wh T(\theta+\rho_1)^3 \tan (\theta+\rho_1) \right|  \frac 12 \cos \theta d\theta \notag \\
&= \int_{-\pi/2+\eps^{1/2}+\rho_1}^{\pi/2-\eps^{1/2}-\rho_1} 
\left| \wh T(\theta)^3 \tan (\theta) \right|  \frac 12 \cos (\theta-\rho_1) d\theta \notag \\
& = 
\int_{-\pi/2+\eps^{1/2}+\rho_1}^{\pi/2-\eps^{1/2}-\rho_1}
\left| \wh T(\theta)^3 \tan (\theta) \right|  \frac 12 
( \cos \theta \cos \rho_1 + \sin \theta\sin\rho_1) d\theta \notag \\
& \leq \int_{-\pi/2+\eps^{1/2}+\rho_1}^{\pi/2-\eps^{1/2}-\rho_1}
\left| \wh T(\theta)^3 \tan (\theta) \right|  \frac 12 
 \cos \theta d\theta \notag \\
& \leq \int_{-\pi/2+\eps^{1/2}+\rho_1}^{\pi/2-\eps^{1/2}-\rho_1} 
\left| \wh T(\theta)^3  \right|    d\theta
\leq \int_{-\pi/2+\eps^{1/2}+\rho_1}^{\pi/2-\eps^{1/2}-\rho_1} 
\left| x(\theta)^3  \right|    d\theta
.\nonumber
\end{align}
It follows from \eqref{d18.9} that
\begin{align*}
|x(\theta)| = O(\eps) (\cos \theta)^{-1} + O(\eps^2)  (\cos \theta)^{-3},
\end{align*}
so
\begin{align*}
|x(\theta)|^3 = O(\eps^3) (\cos \theta)^{-3} + O(\eps^6)  (\cos \theta)^{-9}.
\end{align*}
These bounds and \eqref{d19.5} yield
\begin{align}\label{d22.6}
\E& \left(\left| \wh T(\Theta+\rho)^3 \tan (\Theta+\rho) \right|
\bone_{(-\pi/2+\eps^{1/2} , \pi/2 -\eps^{1/2} -2\rho_1)}(\Theta) \right)\\
&\leq \int_{-\pi/2+\eps^{1/2}+\rho_1}^{\pi/2-\eps^{1/2}-\rho_1} 
\left| x(\theta)^3  \right|    d\theta \notag \\
& \leq \int_{-\pi/2+\eps^{1/2}+\rho_1}^{\pi/2-\eps^{1/2}-\rho_1} 
(O(\eps^3) (\cos \theta)^{-3} + O(\eps^6)  (\cos \theta)^{-9})  d\theta
\notag \\
& \leq O(\eps^3) O(\eps^{-1}) + O(\eps^6) O(\eps^{-4})
= O(\eps^2).\nonumber
\end{align}
The inequality, \eqref{d20.3} and \eqref{d20.2} imply that
\begin{align}\label{d20.4}
\E& \left| \wh T(\Theta+\rho) - T^\eps_1(\pi/2, \Theta )\right|
= O(\eps^2).
\end{align}
This estimate and \eqref{d18.4} give
\begin{align*}
\lim_{\eps\to0} 
\frac {\E T^\eps_1(\pi/2, \Theta )} {\eps^2 \log(1/\eps)}
= h'(\pi/2)h(\pi/2)/2,
\end{align*}
and, therefore, complete the proof of \eqref{MeanT}.

\bigskip
(ii)
Recall definitions and notation from the first part of the proof.
We have
 \begin{align}\label{d19.1}
 \var\left(T^\eps_1(\pi/2, \Theta )\right)
&= \E \left(T^\eps_1(\pi/2, \Theta )^2\right)
- \left( \E T^\eps_1(\pi/2, \Theta ) \right)^2 \\
 &= \mathbb{E}\left((T_1^\epsilon(\pi/2 ,\Theta)-\wh T(\Theta+\rho))^2 \right) 
 + \mathbb{E}\left(\wh T(\Theta+\rho)^2 \right) \notag\\
 &\quad + 2\mathbb{E}\left((T_1^\epsilon(\pi/2 ,\Theta)-\wh T(\Theta+\rho))
\wh T(\Theta+\rho) \right)
- \left( \E T^\eps_1(\pi/2, \Theta ) \right)^2. \notag
 \end{align}
We use \eqref{d18.1} and \eqref{d20.4} in the following two estimates,
 \begin{align}\label{d20.5}
 &\mathbb{E}\left((T_1^\epsilon(\pi/2 ,\Theta)-\wh T(\Theta+\rho))^2 \right) 
\leq O(\eps^{1/2})
\mathbb{E}\left|T_1^\epsilon(\pi/2 ,\Theta)-\wh T(\Theta+\rho) \right|
=O(\eps^{5/2}),\\
&\mathbb{E}\left((T_1^\epsilon(\pi/2 ,\Theta)-\wh T(\Theta+\rho))
\wh T(\Theta+\rho) \right)
\leq O(\eps^{1/2})
\mathbb{E}\left|T_1^\epsilon(\pi/2 ,\Theta)-\wh T(\Theta+\rho) \right|
=O(\eps^{5/2}). \label{d20.6}
 \end{align}
 From \eqref{MeanT}, we obtain
 \begin{equation}\label{d20.7}
 \left(\mathbb{E}T_1^\epsilon(\pi/2, \Theta)\right)^2 =O\big(\eps^{4}\log^2(1/\eps)\big)=o\big(\eps^2\log(1/\eps)\big).
 \end{equation}

We now use the same strategy as in \eqref{d18.2},
\begin{align}\label{d20.8}
\E  \left(\wh T(\Theta+\rho)^2\right) 
&= \int_{-\pi/2}^{\pi/2} 
\left(\wh T(\Theta+\rho_1)^2\right)  \frac 12 \cos \theta d\theta\\
&= \int_{-\pi/2}^{\pi/2-2\rho_1} 
\left(\wh T(\Theta+\rho_1)^2\right)  \frac 12 \cos \theta d\theta
+ \int_{\pi/2-2\rho_1}^{\pi/2} 
\left(\wh T(\Theta+\rho_1)^2\right) \frac 12 \cos \theta d\theta.\nonumber
\end{align}
The second integral can be estimated as follows, using \eqref{d18.1} and \eqref{d19.3},
\begin{align}\label{d20.9}
 \int_{\pi/2-2\rho_1}^{\pi/2} 
\left(\wh T(\Theta+\rho_1)^2\right) \frac 12 \cos \theta d\theta
=
O(\eps) \int_{\pi/2-2\rho_1}^{\pi/2} 
 \cos \theta d\theta
= O(\eps) O(\rho_1^2) = O(\eps^{3}).
\end{align}
For the first integral on the right hand side of \eqref{d20.8}, we 
use the formula 
$\cos (\theta-\rho_1) = \cos \theta \cos \rho_1 + \sin \theta\sin\rho_1$.
and the fact that $\wh T\left(\theta\right)^2$ is an even function,
\begin{align}\label{d20.10}
\int_{-\pi/2}^{\pi/2-2\rho_1} 
&\left(\wh T(\Theta+\rho_1)^2\right) \frac 12 \cos \theta d\theta
= \frac 12 \int_{-\pi/2+\rho_1}^{\pi/2-\rho_1} 
\wh T\left(\theta\right)^2  \cos( \theta-\rho_1) d\theta\\
&= \cos\rho_1 \int_{-\pi/2+\rho_1}^{\pi/2-\rho_1} 
\wh T\left(\theta\right)^2 \frac 12 \cos \theta d\theta
=\cos\rho_1 \int_{-\pi/2}^{\pi/2} 
\wh T\left(\theta\right)^2 \frac 12 \cos \theta d\theta
+ O(\eps^3).\nonumber
\end{align}
The last equality above follows from an estimate similar to the one in \eqref{d20.9}.
We combine \eqref{d20.10} with \eqref{d20.8} and \eqref{d20.9}, and also use \eqref{d19.3}, to obtain 
\begin{align}\label{d20.12}
\lim_{\eps\to0} 
\frac {\E  \left(\wh T(\Theta+\rho)^2\right)} {(\eps^2/2) \log(1/\eps)}
=
\lim_{\eps\to0} 
\frac {1} {(\eps^2/2) \log(1/\eps)}
\int_{-\pi/2}^{\pi/2} 
\wh T\left(\theta\right)^2 \frac 12 \cos \theta d\theta.
\end{align}
The last formula matches 
\eqref{d20.11} except that $\eps $ in \eqref{d20.11} has to be replaced with $|\Gamma^1_\eps(\pi/2)-z|$, which is $h_\eps(\pi/2)+O(\eps^2)$, in view of 
\eqref{d17.1}. It follows from \bht, \eqref{d20.11} and \eqref{d20.12} that
\begin{align*}
\lim_{\eps\to0} 
\frac {\E  \left(\wh T(\Theta+\rho)^2\right)} {(\eps^2/2) \log(1/\eps)}
=
h^2(\pi/2).
\end{align*}
Combining this with \eqref{d19.1}-\eqref{d20.7}
 yields 
\begin{align*}
\lim_{\eps\to0} 
\frac {\var\left(T^\eps_1(\pi/2, \Theta )\right)} {(\eps^2/2) \log(1/\eps)}
=
h^2(\pi/2).
\end{align*}
\end{proof}

\begin{lemma}\label{EstimR} 
The following assertions hold uniformly in $\alpha \in[0,2\pi)$,
\begin{align}\label{MeanEstimDistordu}
\lim_{\eps\rightarrow 0}\frac{\mathbb{E}(R_1^\eps(\alpha))}{(\eps^{2}/2)\log(1/\eps)}&=h'(\alpha)h(\alpha),\\
\label{VarEstimDistordu}
\lim_{\eps\rightarrow 0}\frac{\var(R_1^\eps(\alpha))}{(\eps^2/2)\log(1/\eps)}&=h^2(\alpha).
\end{align}
\end{lemma}

\begin{proof}
The proof proceeds along the same lines as the proof of Lemma \ref{EstimT}.
We will discuss only the changes to that proof that need to be made to accommodate it to the current setting. 

\medskip
(i) The roles of the following objects need to be interchanged:

\begin{enumerate}
\item $\Gamma^0_\eps$ and $\Gamma^1_\eps$, 
\item $f_\eps$ and $g_\eps$; the sign in front  of the function needs to be adjusted, for example, we typically need $1+g_\eps$ and $1-f_\eps$, to match the definitions of $\Gamma^0_\eps$ and $\Gamma^0_\eps$,
\item $|\kappa^0_\eps|$ and $|\kappa^1_\eps|$.
\end{enumerate}

We define $\wh R$ and $R'$ in the way analogous to $\wh T$ and $T'$.

\medskip
(ii) The equation for $x(\theta)$ is analogous to \eqref{d7.2},
\begin{align}\label{d22.2}
x&(\theta)= 
\frac{-a(1+g_\eps(\pi/2)-z_1)+\sqrt{(a^2+1)\kappa^1_\eps(\alpha')^{-2}-(1 +g_\eps(\pi/2)-z_1)^2}}{1+a^2}.
\end{align}
In view of \eqref{d22.1},
the expression under the square root sign in \eqref{d7.2} is equal to
\begin{align*}
&(a^2+1)\kappa^0_\eps(\alpha')^{-2}-(1-f_\eps(\pi/2)-z_1)^2\\
&=( \sin\theta)^{-2}
\left((1+ g_\eps(\pi/2) - z_1)^2 + O(\eps^2)\right)
-(1-f_\eps(\pi/2)-z_1)^2.
\end{align*}
It is easy to see that this quantity is always non-negative for small $\eps>0$.
The analogous expression in \eqref{d22.2} is
\begin{align*}
&(a^2+1)\kappa^1_\eps(\alpha')^{-2}-(1+g_\eps(\pi/2)-z_1)^2\\
&=( \sin\theta)^{-2}
\left((1-f_\eps(\pi/2) - z_1)^2 + O(\eps^2)\right)
-(1+g_\eps(\pi/2)-z_1)^2.
\end{align*}
This quantity is equal to $0$ if
\begin{align}\label{d22.3}
|\sin \theta| = 1  - h_\eps(\pi/2) + O(\eps^2).
\end{align}
Let $\theta_-$ and $\theta_+$ be the two solutions to \eqref{d22.3} in $(-\pi/2, \pi/2)$. Since 
\begin{align*}
\theta_- -( - \pi/2) = O\left(\sqrt{h_\eps(\pi/2)}\right)
= O\left(\eps^{1/2}\right) \quad \text{  and  } \quad
\pi/2-\theta_+ = O\left(\sqrt{h_\eps(\pi/2)}\right)
= O\left(\eps^{1/2}\right),
\end{align*}
we have
\begin{align*}
c_*:=\int_{\theta_-}^{\theta_+} \frac 12 \cos \theta d \theta = 1 - O(\eps). 
\end{align*}
It follows that all integrals of the form 
$\int_{-\pi/2}^{\pi/2} ( \dots ) \frac 12 \cos \theta d \theta$ that appear in the  proof of Lemma \ref{EstimT} should be replaced with the integrals of the form 
$\int_{\theta_-}^{\theta_+} ( \dots ) \frac 1{2c_*} \cos \theta d \theta$
in the present proof. Since $c_* = 1 - O(\eps)$, the extra factor $1/c_*$ will not affect the normalizing constant in
\eqref{MeanEstimDistordu}-\eqref{VarEstimDistordu}
relative to 
\eqref{MeanT}-\eqref{VarT}.

\medskip
(iii)
The last element of the proof of Lemma \ref{EstimT} that needs to be modified is the estimate of $\left| \wh T(\Theta+\rho) - T^\eps_1(\pi/2, \Theta )\right|$. We start by modifying \eqref{d18.5} and \eqref{d20.1}. We divide the interval $(\theta_-,\theta_+)$ 
into two subsets,
\begin{align*}
A_1:=
\left( -\pi/2 +\eps^{1/2} \log^{1/4}(1/\eps),\pi/2 -\eps^{1/2} \log^{1/4}(1/\eps)\right)
\end{align*}
and
\begin{align*}
A_2:=
\left(\theta_-, -\pi/2 +\eps^{1/2} \log^{1/4}(1/\eps)\right)
\cup \left(\pi/2 -\eps^{1/2} \log^{1/4}(1/\eps), \theta_+\right).
\end{align*}
The same geometric analysis as in the proof of Lemma \ref{EstimT} yields
\begin{align*}
\left| \wh R(\theta) - R'(\theta)\right| = O\left( \wh R(\theta)^3 \tan \theta \right)
\end{align*}
for $\theta \in A_1$, and 
\begin{align*}
\left| \wh R(\theta) - R'(\theta)\right| = O\left( \wh R(\theta)^3 \eps^{-1/2} \right)
\end{align*}
for $\theta \in A_2$.
The  analogue of \eqref{d20.3} is
\begin{align}\label{d22.4}
\E& \left| \wh R(\Theta+\rho) - R^\eps_1(\pi/2, \Theta )\right|
= \E \left| \wh R(\Theta+\rho) - R'(\Theta+\rho)\right|\\
&\leq c_4
\E \left(\left| \wh R(\Theta+\rho)^3 \tan (\Theta+\rho) \right|
\bone_{A_1}(\Theta) \right) + c_4
\E \left(\left| \wh R(\Theta+\rho) ^3 \eps^{-1/2}\right|
\bone_{A_2}(\Theta) \right).\notag
\end{align}
The analogue of \eqref{d20.2} is
\begin{align}\label{d22.5}
\E &\left(\left| \wh R(\Theta+\rho) ^3 \eps^{-1/2}\right|
\bone_{A_2}(\Theta) \right)
\leq O(\eps) 
\int_{A_2 } \frac 12 \cos \theta d\theta
= O(\eps) O\left(\eps \log^{1/2}(1/\eps)\right)\\
&=  O\left(\eps^2 \log^{1/2}(1/\eps)\right).\notag
\end{align}
The analogue of \eqref{d22.6} is
\begin{align}\label{d22.7}
\E& \left(\left| \wh R(\Theta+\rho)^3 \tan (\Theta+\rho) \right|
\bone_{A_1}(\Theta) \right)\\
&\leq \int_{-\pi/2+\eps^{1/2} \log^{1/4}(1/\eps)+\rho_1}^{\pi/2-\eps^{1/2} \log^{1/4}(1/\eps)-\rho_1} 
\left| x(\theta)^3  \right|    d\theta \notag \\
& \leq \int_{-\pi/2+\eps^{1/2} \log^{1/4}(1/\eps)+\rho_1}^{\pi/2-\eps^{1/2} \log^{1/4}(1/\eps)-\rho_1} 
(O(\eps^3) (\cos \theta)^{-3} + O(\eps^6)  (\cos \theta)^{-9})  d\theta
\notag \\
& \leq O(\eps^3) O(\eps^{-1} \log^{-1/2}(1/\eps)) + O(\eps^6) O(\eps^{-4}\log^{-1}(1/\eps))
= O(\eps^2).\nonumber
\end{align}
The estimates \eqref{d22.4}, \eqref{d22.5} and \eqref{d22.7} are accurate enough to yield an analogue of \eqref{d20.4}.

\medskip
With these changes, the other steps in the proof  of Lemma \ref{EstimT}
can be easily adjusted to generate a proof of 
\eqref{MeanEstimDistordu}-\eqref{VarEstimDistordu}.
\end{proof}

\begin{lemma}\label{EstimS}
 We have uniformly in $\alpha\in[0,2\pi)$,
\begin{equation*}
\left\vert \mathbb{E}\big(S_1^\eps(\alpha)\big)\right\vert =O(\eps).
\end{equation*}
\end{lemma}

\begin{proof}
Let $\calC$ be the osculating circle  of $\Gamma^0_\eps$ at $\Gamma^0_\eps(\pi/2)$. Note that the osculating circle is defined relative to  $\Gamma^0_\eps$ and not relative to a rotation of $\Gamma^0_\eps$, unlike in the proofs of Lemmas \ref{EstimT} and \ref{EstimR}. 
Let $\calR$ be the rotation about the point $\Gamma^0_\eps(\pi/2)$
such that the center of the circle $\calC_* := \calR(\calC)$ is at a point $(0,z_1)$, with $z_1< 1$.

Suppose that a light ray leaves $\Gamma^0_\eps(\pi/2)$ at an angle $\theta$,
relative to the normal vector to $\Gamma^0_\eps$ at $\Gamma^0_\eps(\pi/2)$.
It follows from Lemma \ref{convex} that there exist $\theta_- \in (-\pi/2, 0)$ and $\theta_+ \in (0,\pi/2)$ such that (i) the light ray does not intersect $\Gamma^1_\eps$ and it intersects $\Gamma^0_\eps$ for $\theta \in A_1 := (-\pi/2, \theta_-) \cup (\theta_+, \pi/2)$, and (ii) the light ray intersects $\Gamma^1_\eps$ before intersecting $\Gamma^0_\eps$ for $\theta \in ( \theta_-,\theta_+)$. 

For $\theta\in A_1$, the light ray  intersects $\Gamma^0_\eps$ at a point  $\wp_\eps(\pi/2,0,\theta) \exp(i (\pi/2+ S_1^\eps(\theta)))$, in the notation of \eqref{o13.2b}; we added $\theta$ to the notation to make dependence on $\theta$ explicit. 
The same light ray will intersect the circle $\calC$ at a point $\wh r(\theta) \exp(i(\pi/2+ \wh S(\theta)))$.
Let $ r_*(\theta) \exp(i(\pi/2+  S_*(\theta)))= \calR\left(\wh r(\theta) \exp(i(\pi/2+ \wh S(\theta)))\right)$. In other words, $ r_*(\theta) \exp(i(\pi/2+  S_*(\theta)))$ represents the point of intersection with the rotated circle $\calC_*$.

Let $\theta_0 = \max(-\theta_-, \theta_+)$
and $A_2 = (-\pi/2, \theta_0) \cup (\theta_0, \pi/2)$. 
 A calculation similar to that in Lemma \ref{critRemain} gives
$\theta_- - (-\pi/2) = O(\eps^{1/2}) $ and $\pi/2 -\theta_+ = O(\eps^{1/2}) $, so
\begin{align}\label{d23.2}
\pi/2-\theta_0 = O(\eps^{1/2}).
\end{align}

By symmetry,
\begin{align}\label{d23.3}
\E \left(S_*(\Theta) \bone_{A_2}(\Theta) \right) = 0.
\end{align}
Elementary geometry shows that 
\begin{align}\label{d23.5}
\lim_{\eps\to0} \sup_{\theta \in A_1}
\left|
\frac {\wh S(\theta)}{2|\pi/2-\theta|} -1\right| =0.
\end{align}
 The angle of rotation for $\calR$ is equal to $|\gamma^0_\eps(\pi/2)|$ and this is of order $O(\eps)$, by \eqref{d4.1}. The radius of $\calC$ is $|\kappa^0_\eps(\pi/2)|^{-1}$ and this is $1 + O(\eps)$, by the same reasoning that gave \eqref{d23.1}.
These observations easily imply that $\left| \wh S(\theta) - S_*(\theta)\right| = O(|\pi/2-\theta| \eps)$, uniformly in $\theta \in A_1$. Hence, using \eqref{d2.2}, \eqref{d23.2} and \eqref{d23.3},
\begin{align}\label{d23.4}
\left|\E \left(\wh S(\Theta) \bone_{A_2}(\Theta) \right)\right|
&=
\left|\E \left(S_*(\Theta) \bone_{A_2}(\Theta) \right)
+
\E \left(\left(\wh S(\Theta)- S_*(\Theta)\right) \bone_{A_2}(\Theta) \right)\right|\\
& =
\left|\E \left(\left(\wh S(\Theta)- S_*(\Theta)\right) \bone_{A_2}(\Theta) \right)\right| \notag \\
& \leq \left| \frac 2 \eps
\int_{(-\pi/2, \theta_0) \cup (\theta_0, \pi/2)} O(|\pi/2-\theta| \eps) \frac 12 \cos \theta d \theta\right|\notag \\
&=  O\left(|\pi/2-\theta_0|^3\right)
= O\left(\eps^{3/2}\right).\notag 
\end{align}

We have the following analogue of \eqref{d18.5},
\begin{align*}
\left| \wh S(\theta) - S_1^\eps(\theta)\right| = O\left( \wh S(\theta)^3 \tan \theta \right),
\end{align*}
which, combined with \eqref{d23.2}, \eqref{d23.5} and \eqref{d23.4}, implies
\begin{align}\label{d24.1}
\left|\E \left( S_1^\eps(\Theta) \bone_{A_2}(\Theta) \right)\right|
&\leq
\left|\E \left(\wh S(\Theta) \bone_{A_2}(\Theta) \right)\right|
+
\left|\E \left(\left(\wh S(\Theta)- S_1^\eps(\Theta)\right) \bone_{A_2}(\Theta) \right)\right|\\
& \leq O\left(\eps^{3/2}\right)
+ \left| \frac 2 \eps \int_{(-\pi/2, \theta_0) \cup (\theta_0, \pi/2)} O(|\pi/2-\theta| ^3) \tan \theta \frac 12 \cos \theta d \theta\right| \notag\\
&= O\left(\eps^{3/2}\right)
+ O\left(|\pi/2-\theta_0|^4/\eps\right)
= O\left(\eps\right).\notag
\end{align}
It remains to estimate 
$\left|\E \left( S_1^\eps(\Theta) \bone_{A_1\setminus A_2}(\Theta) \right)\right|$.

Assume without loss of generality that $\theta_0 =-\theta_-$ so $A_3:=A_1\setminus A_2 = (\theta_+, \theta_0) = (\theta_+, -\theta_-)$.
Lemma \ref{maxDep} implies that if a light ray starting from $\Gamma^0_\eps(\pi/2)$ intersects $\Gamma^1_\eps$ at a point $\Gamma^1_\eps(\alpha)$ then $\pi/2 - 12 \sqrt{\eps} \leq \alpha
\leq \pi/2 + 12 \sqrt{\eps}$.
Let
\begin{align*}
f^+_\eps &= \sup(f_\eps(\alpha): \pi/2 - 12 \sqrt{\eps} \leq \alpha
\leq \pi/2 + 12 \sqrt{\eps}),\\
f^-_\eps &= \inf(f_\eps(\alpha): \pi/2 - 12 \sqrt{\eps} \leq \alpha
\leq \pi/2 + 12 \sqrt{\eps}).
\end{align*}
It follows from \bhth{} that $f^+_\eps - f^-_\eps \leq 24 \sqrt{\eps} \|f'_\eps\|
= O(\eps^{3/2})$.
If a light ray starting from $\Gamma^0_\eps(\pi/2)$ intersects $\Gamma^1_\eps$ then it must intersect the circle $\calC((0,0), 1-f^-_\eps)$ but it cannot intersect $\calC((0,0), 1-f^+_\eps)$.
We will rescale the circles so that we can apply Lemma \ref{critRemain}.
We define $\eps_-$ and $\eps_+$ by
\begin{align*}
1- \eps_- = \frac {1-f^-_\eps}{1+g_\eps(\pi/2)},
\qquad
1- \eps_+ = \frac {1-f^+_\eps}{1+g_\eps(\pi/2)},
\end{align*}
and note that $\eps_+ - \eps_- 
= O(\eps^{3/2})$ because $f^+_\eps - f^-_\eps 
= O(\eps^{3/2})$.
Then a light ray starting from $\Gamma^0_\eps(\pi/2)$ 
at an angle $\theta$ relative to vertical
intersects  $\calC((0,0), 1-f^-_\eps)$  and does not intersect $\calC((0,0), 1-f^+_\eps)$ if and only if a light ray starting from $(0,1)$ 
at an angle $\theta$ relative to vertical
intersects  $\calC((0,0), 1-\eps_-)$  and does not intersect $\calC((0,0), 1-\eps_+)$.
According to Lemma \ref{critRemain}, the angle must be in the range
\begin{align*}
A_4:=
\left(
\arccot\left(
\left(\frac{2 \eps_- -\eps_-^2}{(1-\eps_-)^2}\right)^{1/2}\right),
\arccot\left(
\left(\frac{2 \eps_+ -\eps_+^2}{(1-\eps_+)^2}\right)^{1/2}\right)
\right).
\end{align*}
Since $\eps_+ - \eps_- = O(\eps^{3/2})$, the length of $A_4$ is $O(\eps)$. This implies that $\theta_+ -(-\theta_-) = O(\eps)$. By \eqref{d23.2}, 
$\cos \theta = O(\eps^{1/2})$ for $\theta\in(-\theta_-, \theta_+)$. We combine these observations with Lemma \ref{maxDep} to obtain
\begin{align*}
\left|\E \left( S_1^\eps(\Theta) \bone_{A_1\setminus A_2}(\Theta) \right)\right|
&\leq \left| \frac 2 \eps \int_{ \theta_+}^{-\theta_-} 12 \sqrt{\eps}  \frac 12 \cos \theta d \theta\right|= O\left(\eps\right).
\end{align*}
The lemma follows from this and \eqref{d24.1}.
\end{proof}

Recall definition \eqref{TimeInverse2} of $N^\eps(t)$ and 
for $n\geq0$, let
\begin{align}\label{d25.2}
\mathcal{F}_n^\epsilon & =\sigma(\bal_k^\epsilon,\bs_k^\epsilon ,\, k=1,\cdots ,n),\\
\Delta B_{n+1}^\epsilon &=\mathbb{E}\left(\bal_{n+1}^\epsilon -\bal_n^\epsilon\mid \mathcal{F}_n^\epsilon\right) , \notag \\
 B^\epsilon(n) &= \sum_{k=1}^n \Delta B_k^\epsilon ,  \notag \\
  M^\epsilon(n) &=\bal_n^\eps -B^\eps(n),  \notag \\
  \Delta A_{n+1}^\epsilon &=\mathbb{E}\left(\left(M_{n+1}^\epsilon -M_n^\epsilon\right)^2\mid \mathcal{F}_n^\epsilon\right),  \notag \\
  A^\epsilon(n) &=\sum_{k=1}^n \Delta A_k^\epsilon,  \notag \\
  \chi_{n+1}^\epsilon &= \mathbb{E}\left(\Delta\calT_{n+1}^\epsilon \mid \mathcal{F}_n^\epsilon\right) , \notag \\
\zeta(\eps, t)&= \frac{\pi t}{\eps\log(1/\eps)}, \notag \\
\bM^\eps_t &=  M^\eps(N^\eps(\zeta(\eps, t))), \qquad
\text{  for  } t\geq 0, \notag \\ 
\bB^\eps_t &=  B^\eps(N^\eps(\zeta(\eps, t))), \qquad
\text{  for  } t\geq 0, \notag \\
\bA^\eps_t &=  A^\eps(N^\eps(\zeta(\eps, t))), \qquad
\text{  for  } t\geq 0. \notag 
\end{align}
The expectations in the above definitions exist and are finite because of the estimate given in Lemma \ref{maxDep}.
 For $\alpha\in[0,2\pi)$ and $s=0,1$, let
\begin{align*}
\chi^\epsilon (\alpha,j) = \E \left(\Delta\calT_{n+1}^\epsilon \mid  
\bal^\eps_n = \alpha, \bs^\eps_n = s\right).
\end{align*}

\begin{lemma} \label{NormParam}  
We have uniformly in $\alpha$,
\begin{align}\label{t1a}
\lim_{\eps\rightarrow 0}\frac{\chi^\epsilon (\alpha,1)}{\eps}
&=
\frac{\pi}{2}h(\alpha),\\
\lim_{\eps\rightarrow 0}\frac{\chi^\epsilon (\alpha,0)}{\eps}
&=\frac{\pi}{2}h(\alpha). \label{t1b}
\end{align}
\end{lemma}

\begin{proof}

Let $\beta = 2/3$ and assume that $\eps\in(0,\eps_0)$. Since the curvature of the unit circle is 1, for every $c_1>0$ there exists $c_2>0$ such that the  arc of the circle
\begin{align*}
\left\{(1+g_\eps(\pi/2))e^{i(\pi/2+ \alpha)}: -c_1\eps^\beta \leq \alpha \leq c_1 \eps^\beta\right\}
\end{align*}
lies below the line 
$\{(z_1,z_2) : z_2 = 1+g_\eps(\pi/2) \}$ and above the line
$\{(z_1,z_2) : z_2 = 1+g_\eps(\pi/2) - c_2 \eps^{2\beta}\}$. 
We have assumed that
 $\|g'_\eps\|_\infty = O(\eps)$ so for some $c_3>0$ and all $-c_1\eps^\beta \leq \alpha \leq c_1\eps^\beta$ we have
$|g_\eps(\pi/2+\alpha) - g_\eps(\pi/2)| \leq c_3 \eps^{1+\beta}$. These two observations imply that for some $c_4>0$,
the set
\begin{align*}
\left\{\Gamma^0_\eps(\pi/2+ \alpha): -c_1\eps^\beta \leq \alpha \leq c_1\eps^\beta\right\}
\end{align*}
lies below the line 
$L:=\{(z_1,z_2) : z_2 = 1+g_\eps(\pi/2) + c_4 \eps^{2\beta}\}$ and above the line
$\{(z_1,z_2) : z_2 = 1+g_\eps(\pi/2) - c_4 \eps^{2\beta}\}$.
This implies that if the light ray starts from $\Gamma^1_\eps(\pi/2)$ at time $t=0$, at an angle $\theta$ relative to the vector $(0,1)$
and $-\pi/2+c_5\eps^{1-\beta}\leq \theta \leq \pi/2-c_5\eps^{1-\beta}$ then 
the light ray crosses  $L$  at a point $(z_1, z_2)$ with $|z_1| \leq c_6 \eps^\beta$. This implies that
\begin{align}\label{d30.1}
\frac{h_\eps(\pi/2)- c_4 \eps^{2\beta}}{\cos \theta}
\leq \Delta\calT_{1}^\eps \leq  \frac{h_\eps(\pi/2)+ c_4 \eps^{2\beta}}{\cos \theta}.
\end{align}

Recall the definition of $\gamma^1_\eps(\pi/2)$ stated before  Lemma \ref{angleecart}. In the following formula, we have to shift the angle $\theta$ by $\gamma^1_\eps(\pi/2)$ because $\theta$ is the angle relative to $(0,1)$ in \eqref{d30.1}. It follows from \eqref{d4.2} that $\gamma^1_\eps(\pi/2)=O(\eps)$, so if $-\pi/2+c_5\eps^{1-\beta}\leq \theta \leq \pi/2-c_5\eps^{1-\beta}$ then $-\pi/2+2c_5\eps^{1-\beta}\leq \theta + \gamma^1_\eps(\pi/2) \leq \pi/2-2c_5\eps^{1-\beta}$, for  small $\eps>0$. 
Let $\theta_- = -\pi/2+\eps^{1-\beta}- \gamma^1_\eps(\pi/2)$ 
and
$\theta_+ = \pi/2-\eps^{1-\beta}- \gamma^1_\eps(\pi/2)$.
We use these observations and the estimate  from Lemma \ref{maxDep} to derive the following,
\begin{align*}
\E\left(\Delta\calT_{1}^\eps \right)
&=\E\left(\Delta\calT_{1}^\eps \bone_{(\theta_-, \theta_+)} (\Theta)
\right)
+\E\left(\Delta\calT_{1}^\eps \bone_{(-\pi/2,\theta_-)\cup(\theta_+, \pi/2)} (\Theta)
\right) \\
&=
\int_{\theta_-}^{ \theta_+} 
\frac{ h_\eps(\pi/2)+ O\left( \eps^{2\beta}\right)}
{\cos \left(\theta  + \gamma^1_\eps(\pi/2)\right)}
 \frac 12 \cos \theta d\theta
 +
\int_{(-\pi/2,\theta_-)\cup(\theta_+, \pi/2)}
O(\eps^{1/2}) \frac 12 \cos \theta d\theta\\
&=
\int_{-\pi/2+\eps^{1-\beta} }^{\pi/2-\eps^{1-\beta}} 
\frac{ h_\eps(\pi/2)+ O\left( \eps^{2\beta}\right)}
{\cos \left(\theta \right)}
 \frac 12 \cos \left(\theta - \gamma^1_\eps(\pi/2)\right) d\theta
 + O\left(\eps^{2(1-\beta)+1/2}\right).
\end{align*}
Recall that $\cos (\theta-\gamma^1_\eps(\pi/2)) = \cos \theta \cos \gamma^1_\eps(\pi/2) + \sin \theta\sin\gamma^1_\eps(\pi/2)$. Thus
\begin{align*}
\E\left(\Delta\calT_{1}^\eps \right)
&=
\int_{-\pi/2+\eps^{1-\beta} }^{\pi/2-\eps^{1-\beta}} 
 h_\eps(\pi/2) \cos \gamma^1_\eps(\pi/2)\frac 12  d\theta
+ O\left( \eps^{2\beta}\right)
\int_{-\pi/2+\eps^{1-\beta} }^{\pi/2-\eps^{1-\beta}} 
\frac{\left| \sin \theta\sin\gamma^1_\eps(\pi/2)\right|}
{\cos \left(\theta \right)}
 \frac 12  d\theta\\
&\quad + O\left(\eps^{5/2-2\beta}\right)\\
&=
\frac{\pi}{2} h_\eps(\pi/2)  \left( 1 + O(\eps^{1-\beta})\right)
+ O\left( \eps^{2\beta}\right) O(\eps) O(\eps^{\beta-1})
 + O\left(\eps^{7/6}\right)\\
&=
\frac{\pi}{2} h_\eps(\pi/2)  \left( 1 + O(\eps^{1/3})\right)
+ O\left( \eps^{2}\right) 
 + O\left(\eps^{7/6}\right).
\end{align*}
It follows that
\begin{align*}
\lim_{\eps\rightarrow 0}\frac{\chi^\epsilon (\pi/2,1)}{\eps}
&=
\lim_{\eps\rightarrow 0}\frac{\E\left(\Delta\calT_{1}^\eps \right)}{\eps}
=
\lim_{\eps\rightarrow 0}\frac{\frac{\pi}{2} h_\eps(\pi/2)}{\eps}
=
\frac{\pi}{2}h(\pi/2).
\end{align*}
Our estimates are uniform in $\alpha$ so \eqref{t1a} follows.
The proof of \eqref{t1b} proceeds along similar lines, with only minor modifications, so it is left to the reader.
\end{proof}

\begin{lemma}\label{C3C4C5bis} 
For any $T>0$, 
\begin{align*}
&\lim_{\epsilon\rightarrow 0}
\mathbb{E}\left(\sup_{t\leq T}\left\vert \bal_{N^\eps(\zeta(\eps, t))}^\eps -\bal_{N^\eps(\zeta(\eps, t-))}^\eps\right\vert^2 \right)=0,\\
&\lim_{\epsilon\rightarrow 0}\mathbb{E}(\sup_{t\leq T}\vert \bB^\eps_t-\bB^\eps_{t-}\vert^2 )=0,\\
&\lim_{\epsilon\rightarrow 0}\mathbb{E}(\sup_{t\leq T}\vert \bA^\eps_t-\bA^\eps_{t-}\vert )=0.
\end{align*}
\end{lemma}
\begin{proof}
The quantities $\left\vert\bal_{N^\eps(\zeta(\eps, t))}^\eps -\bal_{N^\eps(\zeta(\eps, t-))}^\eps\right\vert^2 $, $\vert\bB^\eps_t-\bB^\eps_{t-} \vert^2$ and $\vert\bA^\eps_t-\bA^\eps_{t-}\vert $ can be non-zero  only if $\zeta(\eps, t)=\mathcal{T}_k^\eps$ for some $k\geq 1$. It follows from Lemma \ref{maxDep} and definitions of these quantities that they are bounded by $144\eps$. Since this bound is deterministic, the lemma follows.
\end{proof}

\begin{lemma}\label{NbrJumps} There exists a  constant $c$, depending only on $\|g'\|_\infty$,  such that for any $\epsilon\leq \epsilon_0$,
\begin{equation}\label{sautmoyen}
\mathbb{E}\big(N^\epsilon(t)\big)\leq c(t+2\eps)/\epsilon.
\end{equation}
 In particular, $N^\epsilon (t)$ is finite almost-surely.
\end{lemma}

\begin{proof}
Assumption \bho{} implies that the distance between $\Gamma^0_\eps$ and $\Gamma^1_\eps$ is at least $\eps$.
The light ray travels at speed 1 so it takes at least $\eps$ units of time 
between any two consecutive reflections that don't take place on the same piece of the boundary. Thus $n$ crossings from the inner to the outer boundary and $n$ crossings from the outer to the inner boundary must take at least $2n\epsilon$ units of time. 
 It follows that $N^\eps(t) $ is stochastically majorized by $ U(\lceil t/(2\epsilon)\rceil)$, where 
\begin{equation*}
  U(n)=n +\sum_{k=1}^n X^\epsilon_{k},
\end{equation*}
and $X^\eps_k$ are i.i.d. random variables with the geometric distribution (taking values $1,2, \dots$) with parameter $1- \eps\left(4 +6 \|g'\|_\infty\right)$ (see Lemma \ref{probReste}).
Therefore,
\begin{align*}
\mathbb{E}\left(N^\epsilon(t)\right)
&\leq \mathbb{E}\left(U\left(\lceil t/(2\epsilon)\rceil
\right)\right)
=  \lceil t/(2\epsilon)\rceil +  \lceil t/(2\epsilon)\rceil \frac 1 {1-\eps\left(4 +6 \|g'\|_\infty\right)}\\
&= \lceil t/(2\epsilon)\rceil \frac {2-\eps\left(4 +6 \|g'\|_\infty\right)} {1-\eps\left(4 +6 \|g'\|_\infty\right)}
\leq 
\left( \frac t{2\eps} +1\right) \cdot c_1
\leq c(t+2\eps)/\epsilon.
\end{align*}
\end{proof}

\begin{lemma}\label{KolmogDoob} 
Suppose that $\{X_t\}_{t\geq 0}$  is a martingale and let $\tau$ be a  stopping time such that $\E \tau<\infty$. Then for $a>0$,
\begin{equation}\label{KDmaxInegST}
\mathbb{P}\left(\sup_{0\leq t\leq \tau}|X_t|\geq a \right)\leq 
\sup_{s\geq0}\frac{2}{a^2}\mathbb{E}\left(X_{s\land\tau}^2\right).
\end{equation}
\end{lemma} 

\begin{proof}
Let $M_t = X_{t\land \tau}$. By the optional stopping theorem, $\{M_t\}_{t\geq0}$ is a martingale and $\{|M_t|\}_{t\geq 0}$ is a positive submartingale. By Doob's inequality, for any fixed $s>0$,
\begin{align}\label{j2.5}
\P \left(\sup_{0\leq t\leq s}|X_{t\land\tau}|\geq a \right)\leq \frac{2}{a^2}\mathbb{E}\left(X_{s\land\tau}^2\right).
\end{align}
Events $\left\{\sup_{0\leq n\leq k}|X_{n\land\tau}|\geq a \right\}$ converge monotonically to $\left\{\sup_{0\leq n\leq \tau}|X_n|\geq a \right\}$ when $k\to \infty$ so the left hand side of \eqref{j2.5} converges to the left hand side of \eqref{KDmaxInegST}. 
\end{proof}

\begin{lemma}\label{C1bis}
 For any $T>0$, 
$$\sup_{0\leq t\leq T}\left\vert\bB^\eps_t-\int_0^t  h'(\bal_{N^\eps(\zeta(\eps, s))}^\eps)ds\right\vert$$ 
converges to $0$ in probability when $\epsilon\rightarrow 0$.
\end{lemma}
\begin{proof}
Recall that $\zeta(\eps, t)=\frac{\pi t}{\eps\log(1/\eps)}$.
Since $N^\eps(\mathcal{T}_k^\eps)=k $, we get by 
a change of variable,
\begin{align}
\int_0^t h'\left(\bal_{N^\eps(\zeta(\eps, s))}^\eps\right)ds 
&= \frac{\eps\log(1/\eps)}{\pi}
\left(\int_0^{\mathcal{T}^\eps_{N^\eps(\zeta(\eps, t))}}
h'\left(\bal_{N^\eps(s)}^\eps\right)ds 
+ \int_{\mathcal{T}^\eps_{N^\eps(\zeta(\eps, t))}}^{\zeta(\eps, t)}h'\left(\bal_{N^\eps(s)}^\eps\right)ds\right)\notag\\
&= \frac{\eps\log(1/\eps)}{\pi}
\sum_{k=0}^{N^\eps(\zeta(\eps, t))-1}
h'(\bal_{k}^\eps)\Delta\calT_{k+1}^\eps\notag \\
&\quad+ \frac{\eps\log(1/\eps)}{\pi}\big(\zeta(\eps, t)-\mathcal{T}^\eps_{N^\eps(\zeta(\eps, t))}\big)h'\left(\bal_{N^\eps(\zeta(\eps, t))}^\eps\right).\label{j1.2}
\end{align}

From \eqref{o22.8b}, we have for any $t>0$,
\begin{align}\label{j1.1}
\mathcal{T}^\eps_{N^\eps(t)}\leq t
\leq \mathcal{T}^\eps_{N^\eps(t)+1}
=\mathcal{T}^\eps_{N^\eps(t)}+\Delta\calT_{N^\eps(t)+1}^\eps.
\end{align}
Assumption \bho{} and Lemma \ref{maxDep} imply that, a.s., for all $k$ and $\eps\leq \eps_0$,
\begin{align}\label{j2.4}
\Delta \calT^\eps_k
=|Q^\eps(\calT^\eps_{k-1}) - Q^\eps(\calT^\eps_{k})|
\leq 2 |\bal^\eps_{k-1} - \bal^\eps_k| + 6 \eps 
\leq 24 \eps^{1/2} + 6\eps
= O(\eps^{1/2}).
\end{align}
This and \eqref{j1.1} imply that $\zeta(\eps, t)-\mathcal{T}^\eps_{N^\eps(\zeta(\eps, t))}= O(\eps^{1/2})$.
Since $\Vert h'\Vert_\infty<\infty$, we have uniformly in $\alpha\in[0,2\pi)$ and $t\geq 0$,
\begin{align}\label{T1}
\frac{\eps\log(1/\eps)}{\pi}\big(\zeta(\eps, t)-\mathcal{T}^\eps_{N^\eps(\zeta(\eps, t))}\big)
\left\vert h'\left(\alpha\right)\right\vert
=O\left(\eps^{3/2}\log(1/\eps)\right).
\end{align} 
We combine this with  \eqref{j1.2} to see that
\begin{align}\label{DecomposeDrift}
&\left\vert \bB^\eps_t-\int_0^t h'(\bal_{N^\eps(\zeta(\eps, s))}^\eps)ds \right\vert\\ 
&\leq \left\vert \sum_{k=0}^{N^\eps(\zeta(\eps, t))-1} \Delta B^\eps(\bal_{k+1}^\eps,\bs_{k+1}^\eps)-\frac{\eps\log(1/\eps)}{\pi}h'(\bal_k^\eps)\Delta\calT_{k+1}^\eps\right\vert
+O\big(\eps^{\frac{3}{2}}\log(1/\eps)\big)\notag\\
&\leq \left\vert \sum_{k=0}^{N^\eps(\zeta(\eps, t))-1} \Delta B^\eps(\bal_{k+1}^\eps,\bs_{k+1}^\eps)-\frac{\eps\log(1/\eps)}{\pi}h'(\bal_k^\eps)\chi_{k+1}^\eps\right\vert \notag\\
&\quad +\frac{\eps\log(1/\eps)}{\pi}
\left\vert \sum_{k=0}^{N^\eps(\zeta(\eps, t))-1}h'(\bal_k^\eps)(\chi_{k+1}^\eps-\Delta\calT_{k+1}^\eps)\right\vert
+O\big(\eps^{\frac{3}{2}}\log(1/\eps)\big).\notag
\end{align}

We have
\begin{align}\label{j2.1}
&\left\vert \sum_{k=0}^{N^\eps(\zeta(\eps, t))-1} \Delta B^\eps(\bal_{k+1}^\eps,\bs_{k+1}^\eps)-\frac{\eps\log(1/\eps)}{\pi}h'(\bal_k^\eps)\chi_{k+1}^\eps\right\vert\\
&\leq \frac{\eps^{2}\log(1/\eps)}{2} 
\sum_{k=0}^{N^\eps(\zeta(\eps, t))-1}
\left\vert h'(\bal_k^\eps) h(\bal_k^\eps)
\left(\frac{\Delta B^\eps(\bal_k^\eps,\bs_k^\eps)}{(1/2)\eps^2\log(1/\eps)h'(\bal_k^\eps)h(\bal_k^\eps)}
-\frac{2\chi_{k+1}^\epsilon}{\pi\eps h(\bal_k^\eps)}\right)\right\vert\notag\\
&\leq \frac{\eps^{2}\log(1/\eps)}{2} N^\eps(\zeta(\eps, t))
\sup_{\alpha\in\R ,s\in\{0,1\}}
\left\vert h'(\alpha)h(\alpha)\left(\frac{\Delta B^\eps(\alpha,s)}{(1/2)\eps^2\log(1/\eps)h'(\alpha)h(\alpha)}
-\frac{2\chi^\epsilon(\alpha,s)}{\pi\eps h(\alpha)} \right)\right\vert.
\notag
\end{align}
Lemmas \ref{probReste}, \ref{EstimT}, \ref{EstimR} and \ref{EstimS} imply that the following limit holds uniformly in $\alpha $ and $s$,
\begin{align}\label{j2.2}
\lim_{\eps\to0} \frac{\Delta B^\eps(\alpha,s)}{(1/2)\eps^2\log(1/\eps)h'(\alpha)h(\alpha)} =1.
\end{align}
Lemma \ref{NormParam} implies 
that the following limit holds uniformly in $\alpha $ and $s$,
\begin{align*}
\lim_{\eps\to0} \frac{2\chi^\epsilon(\alpha,s)}{\pi\eps h(\alpha)} =1.
\end{align*}
This and \eqref{j2.1}-\eqref{j2.2} imply that
\begin{align}\label{j2.3}
&\left\vert \sum_{k=0}^{N^\eps(\zeta(\eps, t))-1} \Delta B^\eps(\bal_{k+1}^\eps,\bs_{k+1}^\eps)-\frac{\eps\log(1/\eps)}{\pi}h'(\bal_k^\eps)\chi_{k+1}^\eps\right\vert
\leq \frac{\eps^{2}\log(1/\eps)}{2} N^\eps(\zeta(\eps, t))o(1).
\end{align}

Since $t\mapsto N^\eps(t)$ is a non-decreasing function, Lemma \ref{NbrJumps} implies that for some $c_1$ and all $t\leq T$,
\begin{align*}
\E N^\eps(\zeta(\eps, t)) \leq \E N^\eps(\zeta(\eps, T)) \leq c_1 \zeta(\eps, T)/\eps
= c_1\frac{\pi T}{\eps^2\log(1/\eps)}.
\end{align*}
It follows from this and \eqref{j2.3} that
\begin{align}\label{T2}
\lim_{\eps\to0}\mathbb{E}\left(\sup_{0\leq t\leq T}
\left\vert \sum_{k=0}^{N^\eps(\zeta(\eps, t))-1} \Delta B^\eps(\bal_{k+1}^\eps,\bs_{k+1}^\eps)-\frac{\eps\log(1/\eps)}{\pi}h'(\bal_k^\eps)\chi_{k+1}^\eps\right\vert
\right) =0.
\end{align}
This and \eqref{DecomposeDrift} imply that it will suffice to show that
\begin{align}\label{j2.7}
\sup_{0\leq t\leq T}\frac{\eps\log(1/\eps)}{\pi}
\left\vert \sum_{k=0}^{N^\eps(\zeta(\eps, t))-1}h'(\bal_k^\eps)(\chi_{k+1}^\eps-\Delta\calT_{k+1}^\eps)\right\vert \to 0
\end{align}
in probability as $\eps\to0$.

Recall that $\chi_{k+1}^\eps = \mathbb{E}\left(\Delta\calT_{k+1}^\eps\mid \mathcal{F}_{k}^\eps\right)$.
Let $(\mathcal{M}_\eps(n))_{n\geq 0}$ be  defined by $\mathcal{M}_\eps(0)=0$ and for $n\geq 1$ by
\begin{align}\label{j2.6}
\mathcal{M}_\eps(n)&= \frac{\eps\log(1/\eps)}{\pi} \sum_{k=0}^{n-1}h'(\bal_k^\eps)\Big(\Delta\calT_{k+1}^\eps
-\mathbb{E}\big(\Delta\calT_{k+1}^\eps\mid \mathcal{F}_{k}^\eps\big)\Big).
\end{align}
Then $(\mathcal{M}_\eps(n))_{n\geq 0}$ is a martingale and its quadratic variation is given by
\begin{align*}
\big\langle \mathcal{M}_\eps\big\rangle_n&= \frac{\eps^{2}\log^2(1/\eps)}{\pi^2} \sum_{k=0}^{n-1} \left(h'(\bal_k^\eps)\right)^2
\var\left(\Delta\calT_{k+1}^\eps\mid \mathcal{F}_{k}^\eps\right)\\
&\leq \frac{\eps^{2}\log^2(1/\eps)\Vert h'\Vert_\infty^2}{\pi^2} \sum_{k=0}^{n-1} \mathbb{E}\big((\Delta\calT_{k+1}^\eps)^2\mid \mathcal{F}_{k}^\eps\big).
\end{align*}
By \eqref{j2.4}, a.s., for some $c_2$ and all $k$,
$\E \left((\Delta\calT_{k+1}^\eps)^2\mid \mathcal{F}_{k}^\eps\right)
\leq c_2 \eps$.
This implies that for some $c_3$,
\begin{align}\label{Ec}
\big\langle \mathcal{M}_\eps\big\rangle_n \leq c_3\eps^3\log^2(1/\eps)n.
\end{align}
By Lemma \ref{NbrJumps}, $N^\eps(\zeta(\eps, t))$ is a stopping time with a finite expectation so by the optional stopping theorem and estimates \eqref{sautmoyen} and \eqref{Ec}, for any $s>0$,
\begin{align*}
\mathbb{E}\left(\mathcal{M}_\eps^2\big(s\land N^\eps(\zeta(\eps, t))\big)\right)
&=\mathbb{E}\left(\big\langle \mathcal{M}_\eps\big\rangle_{s\land N^\eps(\zeta(\eps, t))}\right)
\leq c_3\eps^3\log^2(1/\eps)\mathbb{E}\left(s\land N^\eps(\zeta(\eps, t))\right)\notag\\
&\leq c_4\eps^3\log^2(1/\eps)\Big(\frac{\pi t}{\eps\log(1/\eps)}\frac{1}{\eps}+2\Big)\overset{\eps\rightarrow 0}{\longrightarrow} 0.
\end{align*} 
We see that the assumptions of  
Lemma \ref{KolmogDoob} are satisfied and we can use that lemma as follows. For any $a>0$,
\begin{align}\label{T3}
\mathbb{P}\left(\sup_{0\leq t\leq T} \vert \mathcal{M}_\eps(N^\epsilon(\zeta(\eps, t)) \vert \geq a \right)
\leq \sup_{s>0}
\frac{2}{a^2}\mathbb{E}\left( \mathcal{M}^2_\eps(N^\epsilon(\zeta(\eps, t))  \right)\overset{\epsilon\rightarrow 0}{\longrightarrow} 0.
\end{align}
The claim \eqref{j2.7} is proved and, therefore, so is the lemma.
\end{proof}

Recall from \eqref{d25.2} that 
$\Delta A_{n+1}^\epsilon =\mathbb{E}\left(\left(M_{n+1}^\epsilon -M_n^\epsilon\right)^2\mid \mathcal{F}_n^\epsilon\right)$.
We will write 
\begin{align*}
\Delta A^\eps(\alpha, s) =\mathbb{E}\left(\left(M_{n+1}^\epsilon -M_n^\epsilon\right)^2\mid \bal_n^\eps = \alpha,\bs_n^\eps=s\right)
\end{align*}
to emphasize the dependence on $\bal_n^\eps$ and $\bs_n^\eps$.
We will also use the self-explanatory notation
$\Delta A^\eps(\bal_n^\eps,\bs_n^\eps)$.

\begin{lemma}\label{j3.1} 
The following limit holds uniformly in $\alpha$ and $s$,
\begin{align}\label{j3.2}
\lim_{\eps\rightarrow 0}\frac{\Delta A^\eps(\alpha,s)}{(\epsilon^2/2)\log(1/\eps)}&=h^2(\alpha).
\end{align}
\end{lemma}
\begin{proof}
Definition \eqref{d25.2} yields
\begin{align*}
\Delta A^\eps(\alpha,s)
&= \E\left(\left(M_{n+1}^\epsilon -M_n^\epsilon\right)^2\mid \bal_n^\eps = \alpha,\bs_n^\eps=s\right)\\
&= \E\left(\left(\bal_{n+1}^\epsilon -\bal_n^\epsilon - \Delta B^\eps(n+1)\right)^2\mid \bal_n^\eps = \alpha,\bs_n^\eps=s\right)\\
&= \var\left(\bal_{n+1}^\epsilon -\bal_n^\epsilon
\mid \bal_n^\eps = \alpha,\bs_n^\eps=s\right).
\end{align*}
For $s=1$, \eqref{VarT} implies that
\begin{align}\label{j3.3}
\lim_{\eps\rightarrow 0}\frac{\Delta A^\eps(\alpha,1)}{(\epsilon^2/2)\log(1/\eps)}
=
\lim_{\eps\rightarrow 0}\frac{\var\left(\bal_{n+1}^\epsilon -\bal_n^\epsilon
\mid \bal_n^\eps = \alpha,\bs_n^\eps=1\right)}{(\epsilon^2/2)\log(1/\eps)}
=
\lim_{\eps\rightarrow 0}\frac{\var(T_1^\eps(\alpha))}{(\epsilon^2/2)\log(1/\eps)}
=h^2(\alpha).
\end{align}

Next we consider the case $s=0$.
Recall \eqref{InducedMCgeneralise}.
By Lemmas \ref{probReste} and \ref{maxDep}, for all $\alpha$, $\E \left(\left( \Lambda_{1}^\epsilon(\alpha) S_{1}^\epsilon(\alpha) \right)^2\right)=O( \eps^2)$.
By Lemmas \ref{probReste} and \ref{EstimS}, for all $\alpha$, $\left|\E \big( \Lambda_{1}^\epsilon(\alpha) S_{1}^\epsilon(\alpha)\big) \right|= O( \eps^2)$.
Hence, in view of Lemma \ref{EstimR},
\begin{align}\label{j3.4}
&\left|\E \big(\Lambda_{1}^\epsilon(\alpha)S_{1}^\epsilon(\alpha) +(1-\Lambda_{1}^\epsilon(\alpha))R_{1}^\epsilon (\alpha)\big)\right|\\
&\quad= (1 + o(1))(\epsilon^2/2)\log(1/\eps) \left|h'(\alpha)h(\alpha)\right|(1-O(\eps))+O(\eps^2).\notag
\end{align}
By Lemma \ref{EstimR} and the Cauchy-Schwartz inequality,
\begin{align*}
&\E \left(\left(\Lambda_{1}^\epsilon(\alpha)S_{1}^\epsilon(\alpha) +(1-\Lambda_{1}^\epsilon(\alpha))R_{1}^\epsilon (\alpha)\right)^2 \right)\\
& = \E \left(\left( \Lambda_{1}^\epsilon(\alpha) S_{1}^\epsilon(\alpha) \right)^2\right)
+ \E \left(\left( (1-\Lambda_{1}^\epsilon(\alpha))R_{1}^\epsilon (\alpha) \right)^2\right)
+ 2 \E \left(\Lambda_{1}^\epsilon(\alpha) S_{1}^\epsilon(\alpha)(1-\Lambda_{1}^\epsilon(\alpha))R_{1}^\epsilon (\alpha) \right)\\
& = O(\eps^2)
+ (1 + O(\eps))(\epsilon^2/2)\log(1/\eps) h^2(\alpha)
+ 0\\
&= (1 + O(\eps))(\epsilon^2/2)\log(1/\eps) h^2(\alpha).
\end{align*}
This and \eqref{j3.4} imply that 
\begin{align*}
\var \left(\Lambda_{1}^\epsilon(\alpha)S_{1}^\epsilon(\alpha) +(1-\Lambda_{1}^\epsilon(\alpha))R_{1}^\epsilon (\alpha)\right)
= (1 + O(\eps))(\epsilon^2/2)\log(1/\eps) h^2(\alpha)+o\big(\eps^2\log(1/\eps)\big).
\end{align*}
Hence,
\begin{align*}
\lim_{\eps\rightarrow 0}\frac{\Delta A^\eps(\alpha,0)}{(\epsilon^2/2)\log(1/\eps)}
&=
\lim_{\eps\rightarrow 0}\frac{\var\left(\bal_{n+1}^\epsilon -\bal_n^\epsilon
\mid \bal_n^\eps = \alpha,\bs_n^\eps=0\right)}{(\epsilon^2/2)\log(1/\eps)}\\
&=
\lim_{\eps\rightarrow 0}\frac{\var\left(\Lambda_{1}^\epsilon(\alpha)S_{1}^\epsilon(\alpha) +(1-\Lambda_{1}^\epsilon(\alpha))R_{1}^\epsilon (\alpha)\right)}{(\epsilon^2/2)\log(1/\eps)}
=h^2(\alpha).
\end{align*}
In  view of \eqref{j3.3}, the proof is complete.
\end{proof}

\begin{lemma}\label{C6bis} 
For any $T>0$,
$\sup_{t\leq T}\left\vert \bA^\epsilon_t-\int_0^t h(\bal_{N^\eps(\zeta(\eps, s))}^\eps)ds\right\vert $ converges to $0$ in probability when $\epsilon\rightarrow 0$.
\end{lemma}
\begin{proof}
The proof is the same as that for Lemma \ref{C1bis}, except for the following changes.

(i) $\Delta B^\eps(\alpha,s)$ should be replaced with $\Delta A^\eps(\alpha,s)$.

(ii) $h'(\bal_{N^\eps(\zeta(\eps, s))}^\eps)$ should be replaced with $h(\bal_{N^\eps(\zeta(\eps, s))}^\eps)$.

(iii) Formula \eqref{j2.2} should be replaced with the following consequence of Lemma \ref{j3.1},
\begin{align*}
\lim_{\eps\to0} \frac{\Delta A^\eps(\alpha,s)}{(1/2)\eps^2\log(1/\eps)h^2(\alpha)} =1.
\end{align*}
\end{proof}

\begin{proof}[Proof of Theorem \ref{MainThm3} ]

We will prove that
 processes $\left\{\bal^\eps(N^\eps(\zeta(\eps, t))) , t\geq 0\right\}$
converge in law to $X$ in the Skorokhod topology as $\epsilon$ goes to $0$, where $X$ 
solves the stochastic differential equation \eqref{j4.1}.

The above claim implies easily Theorem \ref{MainThm3}
because $\bbet^\eps(\mathcal{T}_k^\eps) =\bal_k^\eps$ and the jumps of $\bal^\epsilon$ are uniformly bounded by a quantity going to $0$ when $\eps\rightarrow 0$, by Lemma \ref{maxDep}.

To prove the claim stated at the beginning of the proof,
we will apply \cite[Thm. 4.1, Ch. 7]{EthierKurtz}. We start with a dictionary translating our notation to that in \cite{EthierKurtz}. In the following list, our symbol is written to the left of the arrow and the corresponding symbol used in \cite{EthierKurtz} is written to the right of the arrow. Note that our family of processes is indexed by a continuous parameter $\eps\in(0,1/2)$ and the corresponding family of processes in \cite{EthierKurtz} is indexed by a discrete parameter $n$. Standard arguments show that nevertheless \cite[Thm. 4.1, Ch. 7]{EthierKurtz} applies in our setting.
\begin{align*}
\bal^\eps(N^\eps(\zeta(\eps, t))) &\La X_n, \\
\bB^\eps_t &\La B_n,\\
\bM^\eps_t &\La M_n,\\
\bA^\eps_t &\La A_n, \\
 h'(\bal^\eps(N^\eps(\zeta(\eps, s)))) &\La b(X_n(s)),\\
 h^2(\bal^\eps(N^\eps(\zeta(\eps, s)))) &\La a(X_n(s)).
\end{align*}

Many of the assumptions of \cite[Thm. 4.1, Ch. 7]{EthierKurtz} are clearly satisfied and, therefore, we will not discuss them explicitly. For example, our assumptions on the smoothness of $h$ are so strong that the martingale problem corresponding to \eqref{j4.1} is well-posed.

We will now review the crucial assumptions of \cite[Thm. 4.1, Ch. 7]{EthierKurtz}. 

Recall notation introduced in \eqref{d25.2}.
Let $\calG^\eps_t = \sigma((\bal^\eps(N^\eps(\zeta(\eps, s))),\bB^\eps_s,\bA^\eps_s), s\leq t)$.
It was proved in Lemma \ref{maxDep} that the absolute value of a jump of $\bal^\eps_n$ is bounded by $12 \sqrt{\eps}$.
The same bound applies to jumps of $B^\eps(n)$  and, therefore, the absolute value of a jump of $M^\eps(n)$ is bounded by $24 \sqrt{\eps}$, a.s.
Hence, it is easy to see from the definition \eqref{d25.2} that $M^\eps(n)$ is a martingale. Let $\tau^\eps_r = \inf\{n: |M^\eps(n)| \geq r\}$ and note that 
$|M^\eps(n)| \leq r+24 \sqrt{\eps}$ for $n\leq \tau^\eps_r$. This easily implies that the optional stopping theorem applies to the martingale $M^\eps(n)$ at the stopping time $N^\eps(\zeta(\eps, t)) \land \tau^\eps_r$, for every $t$ and $r$. This in turn implies that $\bM^\eps_t$ is a $\calG^\eps_t$-local martingale. A similar argument shows that $(\bM^\eps_t)^2 - \bA^\eps_t$ is a $\calG^\eps_t$-local martingale.
We have verified the assumption that processes defined in (4.1) and (4.2) in 
\cite[Ch. 7]{EthierKurtz} are local martingales.

Assumptions (4.3)-(4.5) in 
\cite[Ch. 7]{EthierKurtz} are satisfied due to
Lemma \ref{C3C4C5bis}.
Assumptions (4.6) and (4.7) in 
\cite[Ch. 7]{EthierKurtz} are satisfied due to
Lemmas \ref{C1bis} and \ref{C6bis}.

We have shown that the assumptions of \cite[Thm. 4.1, Ch. 7]{EthierKurtz} are satisfied. Therefore, we may conclude that $\left\{\bal^\eps(N^\eps(\zeta(\eps, t))) , t\geq 0\right\}$
converge in law to $X$ in the Skorokhod topology as $\epsilon$ goes to $0$, where $X$ 
solves the stochastic differential equation \eqref{j4.1}.
We have already pointed out that this implies  Theorem \ref{MainThm3}.
\end{proof}

\bibliographystyle{plain}
\bibliography{tire}

\end{document}